\documentclass{amsart}

\def\pasdegrille{\let\grille = \pasgrille}
\def\ecriture#1#2{\setbox1=\hbox{#1}
\dimen1= \wd1
\dimen2=\ht1
\dimen3=\dp1
\grille #2 \box1 }
\def\aat#1#2#3{
\divide \dimen1 by 48
\dimen3=\dimen1
\multiply \dimen1 by #1
\advance \dimen1 by -\dimen3
\divide \dimen1 by 101
\multiply \dimen1 by 100
\divide \dimen2 by \count11
\multiply \dimen2 by #2
\setbox0=\hbox{#3}\ht0=0pt\dp0=0pt
  \rlap{\kern\dimen1 \vbox to0pt{\kern-\dimen2\box0\vss}}\dimen1= \wd1
\dimen2=\ht1}
\def\pasgrille{
\count12= \dimen1
\divide \count12 by 50
\divide \dimen2 by \count12
\count11 =\dimen2
\
\divide \dimen1 by 48
\setlength{\unitlength}{\dimen1}
\smash{\rlap{\ }}
\dimen1= \wd1
\dimen2=\ht1
}
\def\grille{
\count12= \dimen1
\divide \count12 by 50
\divide \dimen2 by \count12
\count11 =\dimen2
\
\divide \dimen1 by 48
\setlength{\unitlength}{\dimen1}
\smash{\rlap{\graphpaper[1](0,0)(50, \count11)}}
\dimen1= \wd1
\dimen2=\ht1
}

\pasdegrille
\usepackage{amssymb}
\usepackage{amsmath, amsthm, amsopn, amsfonts}
\usepackage[dvips]{graphicx}
\usepackage{graphpap}
\usepackage[english, francais]{babel}

\newtheorem{theorem}{Theorem}[section]
\newtheorem{lemma}{Lemma}[section]

\newtheorem{proposition}{Proposition}[section]
\newtheorem{corollary}{Corollary}[section]
\newtheorem{remark}{Remark}[section]

\numberwithin{equation}{section}
\numberwithin{proposition}{section}
\def\11{{\rm 1~\hspace{-1.4ex}l} }

\newcommand{\beq}{\begin{equation}}
\newcommand{\eeq}{\end{equation}}
\newcommand{\ben}{\begin{eqnarray}}
\newcommand{\een}{\end{eqnarray}}
\newcommand{\beno}{\begin{eqnarray*}}
\newcommand{\eeno}{\end{eqnarray*}}
\def\R{\mathbb R}

\def\Z{\mathbb Z}
\def\N{\mathbb N}

\def\ba{\begin{aligned}}
\def\ea{\end{aligned}}
\def\be{\begin{equation}}
\def\ee{\end{equation}}
\def\ben{\begin{align*}}
\def\enn{\end{align*}}
\def\R{\mathbb{R}}

\def\Z{\mathbb{Z}}

\begin{document}
\selectlanguage{english}
\title[Strichartz estimates for wave equation with inverse square
potential] {Strichartz estimates for wave equation with inverse
square potential}


\author{Changxing Miao}
\address{Institute of Applied Physics and Computational Mathematics, P. O. Box 8009, Beijing, China, 100088}
\email{miao\_changxing@iapcm.ac.cn}

\author{Junyong Zhang}
\address{Department of Mathematics, Beijing Institute of Technology, Beijing,
China, 100081} \email{zhang\_junyong@bit.edu.cn}

\author{Jiqiang Zheng}
\address{The Graduate School of China Academy of Engineering Physics, P. O. Box 2101, Beijing, China, 100088}
\email{zhengjiqiang@gmail.com}

\begin{abstract}
In this paper, we study the Strichartz-type estimates of the
solution for the linear wave equation with inverse square potential.
Assuming the initial data possesses additional angular regularity,
especially the radial initial data, the range of admissible pairs is
improved. As an application, we show the global well-posedness of
the semi-linear  wave equation with inverse-square potential
$\partial_t^2 u-\Delta u+\frac{a}{|x|^2}u=\pm|u|^{p-1}u$ for power
$p$ being in some regime when the initial data are radial. This
result extends the well-posedness result in Planchon, Stalker, and
Tahvildar-Zadeh.
\end{abstract}
\maketitle
\selectlanguage{english}
\tableofcontents


\section{Introduction and Statement of Main Result}
The aim of this paper is to study the $L^q_t(L^r_x)$-type estimates
of the solution for the linear wave equation perturbed by an inverse
square potential. More precisely, we shall consider the following
wave equation with the inverse square potential
\begin{equation}\label{1.1}
\begin{cases}
\partial_t^2 u-\Delta u+\frac{a}{|x|^2}u=0,\qquad
(t,x)\in\R\times\R^n, ~a\in\R,\\ u(t,x)|_{t=0}=u_0(x),\quad
\partial_tu(t,x)|_{t=0}=u_1(x).
\end{cases}
\end{equation}
The scale-covariance elliptic operator $-\Delta+\frac{a}{|x|^2}$
appearing in \eqref{1.1} plays a key role  in many problems of
physics and geometry. The heat and Schr\"odinger flows for the
elliptic operator $-\Delta+\frac{a}{|x|^2}$ have been studied in the
theory of combustion \cite{LZ}, and in quantum mechanics
\cite{KSWW}.  The equation \eqref{1.1} arises in the study of the
wave propagation on conic manifolds \cite{CT}. We refer the readers
to \cite{BPSS,BPSS1,PSS,PSS1} and references therein.\vspace{0.2cm}

It is well known that Strichartz-type estimates are crucial in
handling local and global well-posedness problems of nonlinear
dispersive equations. Along this way, Planchon, Stalker, and
Tahvildar-Zadeh \cite{PSS} first showed a generalized Strichartz
estimates for the equation \eqref{1.1} with radial initial data.
Thereafter, Burq,  Planchon, Stalker, and Tahvildar-Zadeh
\cite{BPSS} removed the radially symmetric assumption in \cite{PSS}
and then obtained some well-posedness results for the semi-linear
wave equation with inverse-square potential. The range of the
admissible exponents $(q,r)$ for the Strichartz estimates of
\eqref{1.1} obtained in \cite{BPSS,PSS} is restricted under
$\frac2q\leq (n-1)(\frac12-\frac1r)$, which is the same as that of
the linear wave equation without potential.  Sterbenz and Rodnianski
\cite{Sterbenz} improved the range of the ``classical" admissible
exponents $(q,r)$ for the linear wave equation with no potential by
compensating a small loss of angular regularity.\vspace{0.2cm}

 In this paper, we
are devoted to study the Strichartz estimates of the solution of the
equation \eqref{1.1}. By employing the asymptotic behavior of the
Bessel function and some fine estimates of Hankel transform, we
improve the range of the admissible pairs $(q,r)$ in \cite{BPSS,PSS}
by compensating a small loss of angular regularity. The machinery we
employ here is mainly based on the spherical harmonics expansion and
some properties of Hankel transform. As an application of the
Strichartz estimates, we obtain well-posedness of \eqref{1.1}
perturbed by nonlinearity $|u|^{p-1}u$ with power
$p_h<p<p_{\text{conf}}$ (defined below) in the radial case, which
extends the well-posedness result in Planchon et al. \cite{PSS}.
\vspace{0.2cm}

Before stating our main theorems, we need some notations. We say the
pair  $(q,r)\in\Lambda$, if $q,r\geq2$, and satisfy
$$\frac1q\geq\frac{n-1}2(\frac12-\frac1r)\quad \text{and}\quad
\frac1q<(n-1)(\frac12-\frac1r).$$ Set the infinitesimal generators
of the rotations on Euclidean space:
\begin{equation*}
\Omega_{j,k}:=x_j\partial_k-x_k\partial_j,
\end{equation*}
and define for $s\in\R$,
\begin{equation*}
\Delta_\theta:=\sum_{j<k}\Omega_{j,k}^2,\quad
|\Omega|^s=(-\Delta_{\theta})^{\frac{s}2}.
\end{equation*}

\begin{theorem}\label{thm} Let $u$ be a  solution of the equation \eqref{1.1} with
$a>\frac1{(n-1)^2}-\frac{(n-2)^2}4$. For any $\epsilon>0$ and
$0<s<1+\min\big\{\frac{n-2}{2}, \sqrt{(\frac{n-2}2)^2+a}\big\}$,

$\bullet$ if $n\geq4$, then
\begin{equation}\label{1.2}
\|u(t,x)\|_{L^q_tL^r_x}\leq
C_\epsilon\Big(\|\langle\Omega\rangle^{\bar{s}} u_0\|_{\dot
H^s}+\|\langle\Omega\rangle^{\bar{s}} u_1\|_{\dot H^{s-1}}\Big),
\end{equation}
where $(q,r)\in \Lambda$, and
$$\bar{s}=(1+\epsilon)\Big(\frac2q-(n-1)\big(\frac12-\frac1r\big)\Big)\quad
\text{and}\quad s=n\Big(\frac12-\frac1r\Big)-\frac1q;$$

$\bullet$ if $n=3$, then
\begin{equation}\label{1.3}
\|u(t,x)\|_{L^q_tL^r_x}\leq
C_\epsilon\Big(\|\langle\Omega\rangle^{\bar{s}} u_0\|_{\dot
H^s}+\|\langle\Omega\rangle^{\bar{s}} u_1\|_{\dot H^{s-1}}\Big),
\end{equation}
where $q\neq 2, (q,r)\in \Lambda$, and
$$\bar{s}=(2+\epsilon)\Big(\frac1q-\big(\frac12-\frac1r\big)\Big)\quad
\text{and}\quad s=3\Big(\frac12-\frac1r\Big)-\frac1q.$$ In addition,
the following estimate holds for $r>4$ and
$s=3(\frac12-\frac1r)-\frac12$,
\begin{equation}\label{1.4}
\|u(t,x)\|_{L^2_tL^r_x}\leq
C_\epsilon\Big(\|\langle\Omega\rangle^{\bar{s}(r)} u_0\|_{\dot
H^s}+\|\langle\Omega\rangle^{\bar{s}(r)} u_1\|_{\dot H^{s-1}}\Big),
\end{equation}
where $\bar{s}(r)=1-\frac2r$ with $r\neq\infty$.
\end{theorem}
\begin{remark}

i$)$. We remark that some of admissible pairs $(q,r)$ in Theorem 1.1
are out of the region $ACDO$ or $ACO$ $($in the following figures$)$
obtained in \cite {BPSS,PSS}.

ii$)$. Our restriction $a>a_n:=\frac1{(n-1)^2}-\frac{(n-2)^2}4$ is
to extend the the range of $(q,r)$ as widely as possible. We remark
that $a_3=0$ and $a_n<0$ for $n\geq4$. Therefore, we recover the
result of Theorem 1.5 in Sterbenz \cite{Sterbenz}, which considers
$a=0$ and $n\geq4$.

iii$)$. In the extended region $\Lambda$ $($see the below
figures$)$, the loss of angular regularity is
$\bar{s}=(1+\epsilon)(\frac2q-(n-1)(\frac12-\frac1r))$. When $n=3$,
 the loss of angular
regularity in the line $BC$ is $\bar{s}(r)>\bar{s}$, since the
Strichartz estimate fails at the endpoint $(q,r,n)=(2,\infty,3)$. It
seems that the methods we use here are not available to obtain such
estimate at endpoint since Lemma \ref{square-expression} and Lemma
\ref{square-expression2} fail at $r=\infty$. And one might need the
wave packet method of Wolff \cite{Wolff} and the argument in Tao
\cite{Tao4} to obtain the Strichartz estimate at the endpoint
$(q,r,n)=(2,\infty,3)$ with some loss of angular regularity.

\end{remark}
%

\begin{figure}[ht]
\begin{center}
$$\ecriture{\includegraphics[width=6cm]{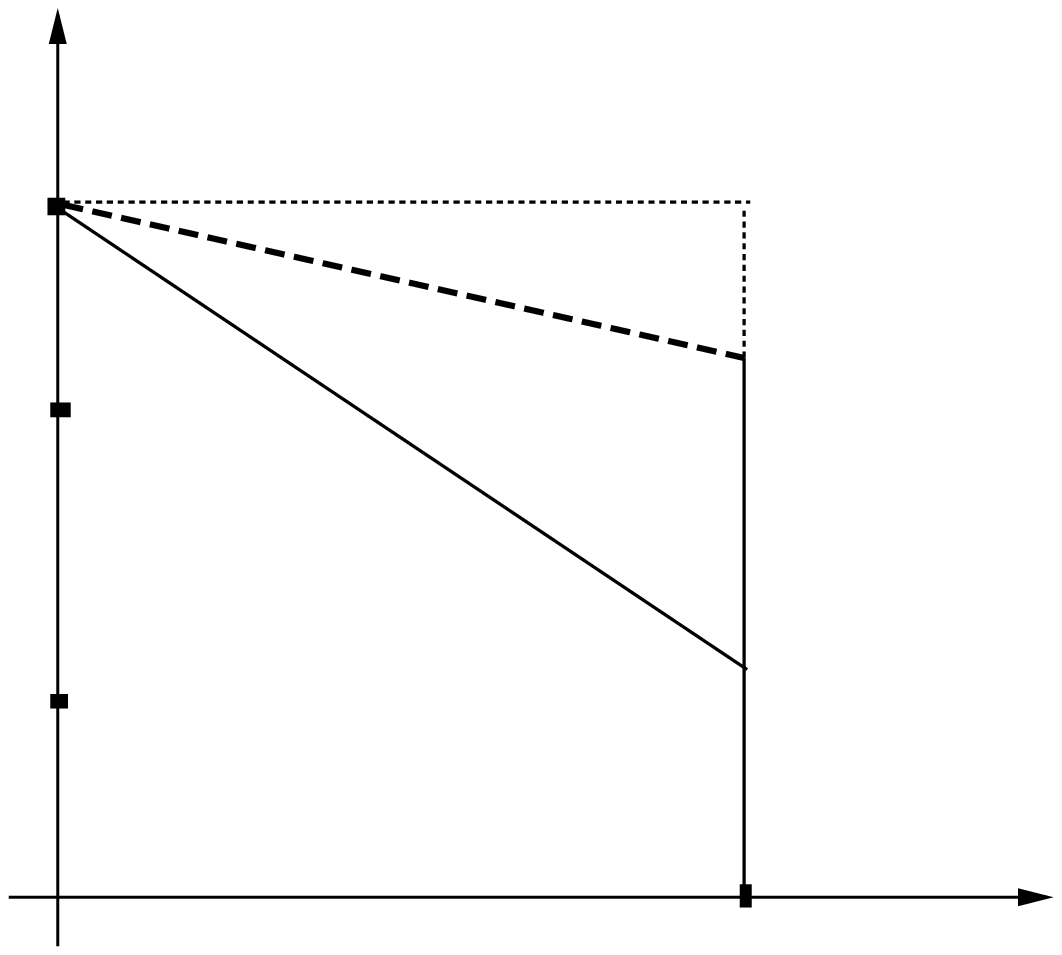}}
{\aat{1}{1}{\small
O}\aat{-4}{12}{$\tfrac{n-3}{2(n-1)}$}\aat{-4}{25}{$\tfrac{n-2}{2(n-1)}$}
\aat{-2}{34}{$\tfrac12$}\aat{-2}{42}{$\tfrac1r$} \aat{6}{37}{\small
A} \aat{16}{-1}{\small
$n>3$}\aat{34}{-1}{$\tfrac12$}\aat{48}{-1}{$\tfrac1q$}\aat{37}{6}{\small
D}\aat{35}{13}{\small C}\aat{35}{27}{\small B}\aat{30}{20}{\small
$\Lambda$}}\qquad \ecriture{\includegraphics[width=6cm]{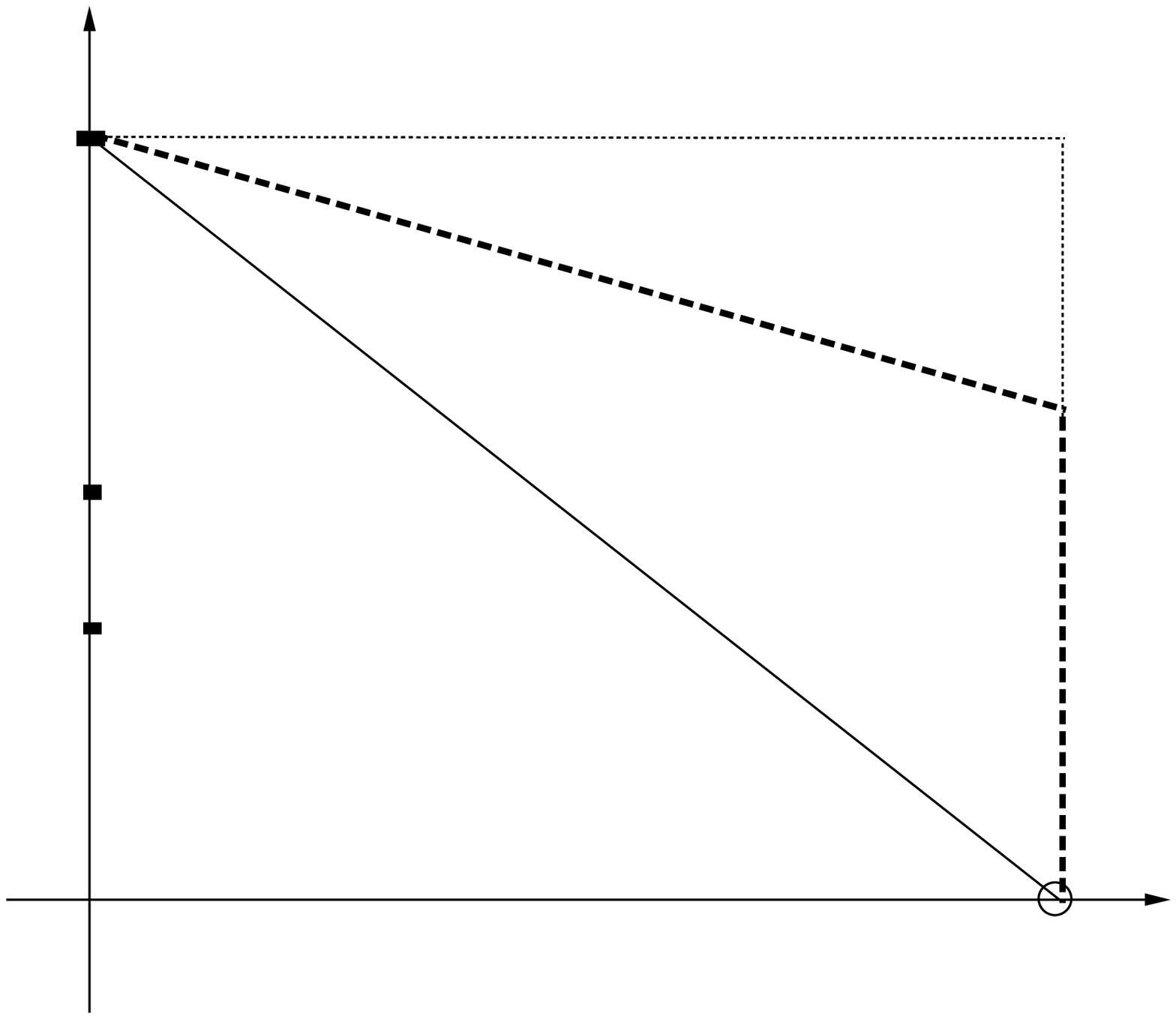}}
{\aat{1}{2}{\small
O}\aat{0}{16}{$\tfrac16$}\aat{0}{22}{$\tfrac14$}\aat{0}{36}{$\tfrac12$}\aat{0}{42}{$\tfrac1r$}
\aat{7}{38}{\small A} \aat{22}{3}{\small
$n=3$}\aat{44}{2}{$\tfrac12$}\aat{48}{2}{$\tfrac1q$}\aat{46}{6}{\small
C}\aat{46}{25}{\small B}\aat{30}{20}{\small $\Lambda$}}$$
\end{center}
\end{figure}

\vspace{0.2cm}
 As a consequence of Theorem \ref{thm} and Corollary
3.9 in \cite{PSS}, we have the following Strichartz estimates for
radial initial data:
\begin{corollary}\label{cor} Let $n\geq3$ and $s<\frac n2$. Suppose $(u_0, u_1)$ are radial functions, then for $q,r\geq2, \frac 1q<(n-1)(\frac12-\frac1r)$
and $s=n(\frac12-\frac1r)-\frac1q$, the solution $u$ of the equation
\eqref{1.1} with $a>\frac1{(n-1)^2}-\frac{(n-2)^2}4$ satisfies
\begin{equation}\label{1.5}
\|u(t,x)\|_{L^q_tL^r_x}\leq C\Big(\|u_0\|_{\dot H^s}+\|u_1\|_{\dot
H^{s-1}}\Big).
\end{equation}
\end{corollary}
As an application, we obtain some well-posedness result of the
following semi-linear wave equation,
\begin{equation}\label{1.6}
\begin{cases}
\partial_t^2 u-\Delta u+\frac{a}{|x|^2}u=\pm|u|^{p-1}u,\qquad
(t,x)\in\R\times\R^n, ~a\in\R,\\
u(t,x)|_{t=0}=u_0(x),\quad
\partial_tu(t,x)|_{t=0}=u_1(x).
\end{cases}
\end{equation}

In the case of the semi-linear wave equation without potential (i.e.
$a=0$), there are many exciting results on the global existence and
blow-up. We refer the readers to \cite{LS,Sogge2} and references
therein. While for the equation \eqref{1.6} with $p\geq
p_{\text{conf}}:=1+\frac4{n-1}$ and $n\geq3$, Planchon et al.
\cite{PSS} established the global existence when the radial initial
data is small in $\dot H^{s_c}\times \dot H^{s_c-1}$-norm with
$s_c:=\frac{n}2-\frac2{p-1}$. Thereafter, Burq et al. \cite{BPSS}
removed the radially symmetric assumption on the initial data. As a
consequence of Theorem \ref{thm}, we prove the global existence of
the solution to the equation \eqref{1.6} with
$p_{\text{h}}:=1+\frac{4n}{(n+1)(n-1)}<p<p_{\text{conf}}$ for small
radial initial data $(u_0,u_1)\in \dot H^{s_c}\times \dot H^{s_c-1}
$.

\begin{theorem}\label{thm1} Let $n\geq3$ and $p_{\text{h}}<p<p_{\text{conf}}$.
Let $q_0=(p-1)(n+1)/2$, $r_0=(n+1)(p-1)/(2p)$, and
\begin{equation}\label{1.7}
a>\max\Big\{\frac1{(n-1)^2}-\frac{(n-2)^2}4, \frac n{q_0}\Big(\frac
n{q_0}-n+2\Big),\big(\frac n{r_0}-n\big)\big(\frac
n{r_0}-2\big)\Big\}. \end{equation} Assume $(u_0,u_1)$ are radial
functions and there is a small constant $\epsilon(p)$ such that
\begin{equation}\label{1.8}
\|u_0\|_{\dot H^{s_c}}+\|u_1\|_{\dot H^{s_c-1}}<\epsilon(p),
\end{equation}
then there exists a unique global solution $u$ to \eqref{1.6}
satisfying
\begin{equation}\label{1.9} u\in C_t(\R;\dot H^{s_c})\cap
L^{q_0}_{t,x}(\R\times\R^n).\end{equation}
\end{theorem}
\begin{remark}
i$)$. The above result extends the well-posedness result in
\cite{PSS} from $p\geq p_{\text{conf}}$ to $p_{\text{h}}<p<
p_{\text{conf}}$.

ii$)$. We remark that the $L^1$-bound of the operator
$\mathcal{K}_{\lambda,\nu}^0$ defined below is the source of our
constraint to $p>p_h$. Inspired by the arguments in
Lindblad-Sogge\cite{LS1,Sogge2} for the usual semi-linear wave
equation, if we want to extend the above result to $p>p_{\text{c}}$,
one needs to explore new inhomogeneous Strichartz estimates since
the operator $\mathcal{K}_{\lambda,\nu}^0$ is not known as a bounded
operator on $L^1$. Here $p_c$ is the positive root of
$(n-1)p_c^2-(n+1)p_c-2=0$, and $p_c$ is called the Strauss's index.
\end{remark}

This paper is organized as follows: In the section 2, we revisit the
property of the Bessel functions, harmonic projection operator, and
the Hankel transform associated with $-\Delta+\frac{a}{|x|^2}$.
Section 3 is devoted to establishing some estimates of the Hankel
transform. In Section 4, we use the previous estimates to prove
Theorem \ref{thm}. We show Theorem \ref{thm1}  in Section 5. In the
appendix, we sketch the proof of Lemma \ref{square-expression} by
using a weak-type $(1,1)$ estimate for the multiplier operators with
respect to the Hankel transform.\vspace{0.2cm}

Finally, we conclude this section by giving some notations which
will be used throughout this paper. We use $A\lesssim B$ to denote
the statement that $A\leq CB$ for some large constant $C$ which may
vary from line to line and depend on various parameters, and
similarly use $A\ll B$ to denote the statement $A\leq C^{-1} B$. We
employ $A\sim B$ to denote the statement that $A\lesssim B\lesssim
A$. If the constant $C$ depends on a special parameter other than
the above, we shall denote it explicitly by subscripts. We briefly
write $A+\epsilon$ as $A+$ for $0<\epsilon\ll1$. Throughout this
paper, pairs of conjugate indices are written as $p, p'$, where
$\frac{1}p+\frac1{p'}=1$ with $1\leq p\leq\infty$.\vspace{0.2cm}

\section{Preliminary}

In this section, we provide some standard facts about the Hankel
transform and the Bessel functions. We use the oscillatory integral
argument to show the asymptotic behavior of the derivative of the
Bessel function. The Littlewood-Paley theorems associated to the
Hankel transform are collected in this section. Finally we prove a
Stirchartz estimate for unit frequency by making use of some results
in \cite{BPSS}.\vspace{0.2cm}

\subsection{Spherical harmonic expansions and the Bessel functions}
We begin with the spherical harmonics expansion formula. For more
details, we refer to Stein-Weiss \cite{SW}. Let
\begin{equation}\label{2.1}
\xi=\rho \omega \quad\text{and}\quad x=r\theta\quad\text{with}\quad
\omega,\theta\in\mathbb{S}^{n-1}.
\end{equation}

For any $g\in L^2(\R^n)$, we have the expansion formula
\begin{equation*}
g(x)=\sum_{k=0}^{\infty}\sum_{\ell=1}^{d(k)}a_{k,\ell}(r)Y_{k,\ell}(\theta)
\end{equation*}
where
\begin{equation*}
\{Y_{k,1},\ldots, Y_{k,d(k)}\}
\end{equation*}
is the orthogonal basis of the spherical harmonic space of degree
$k$ on $\mathbb{S}^{n-1}$, called $\mathcal{H}^{k}$, with the
dimension
\begin{equation*}
d(0)=1\quad\text{and}\quad
d(k)=\frac{2k+n-2}{k}C^{k-1}_{n+k-3}\simeq \langle k\rangle^{n-2}.
\end{equation*}
We remark that for $n=2$, the dimension of $\mathcal{H}^{k}$ is a
constant independent of $k$. We have the orthogonal decomposition
\begin{equation*}
L^2(\mathbb{S}^{n-1})=\bigoplus_{k=0}^\infty \mathcal{H}^{k}.
\end{equation*}
This gives by orthogonality
\begin{equation}\label{2.2}
\|g(x)\|_{L^2_\theta}=\|a_{k,\ell}(r)\|_{\ell^2_{k,\ell}}.
\end{equation}
By Theorem 3.10 in \cite{SW}, we have the Hankel transforms formula
\begin{equation}\label{2.3}
\hat{g}(\rho\omega)=\sum_{k=0}^{\infty}\sum_{\ell=1}^{d(k)}2\pi
i^{k}Y_{k,\ell}(\omega)\rho^{-\frac{n-2}2}\int_0^\infty
J_{k+\frac{n-2}2}(2\pi r\rho)a_{k,\ell}(r)r^{\frac n2}\mathrm{d}r,
\end{equation}
here the Bessel function $J_k(r)$ of order $k$ is defined by
\begin{equation*}
J_k(r)=\frac{(r/2)^k}{\Gamma(k+\frac12)\Gamma(1/2)}\int_{-1}^{1}e^{isr}(1-s^2)^{(2k-1)/2}\mathrm{d
}s\quad\text{with}~ k>-\frac12~\text{and}~ r>0.
\end{equation*}

A simple computation gives the estimates
\begin{equation}\label{2.4}
|J_k(r)|\leq
\frac{Cr^k}{2^k\Gamma(k+\frac12)\Gamma(1/2)}\big(1+\frac1{k+1/2}\big),
\end{equation}
and
\begin{equation}\label{2.5} |J'_k(r)|\leq
\frac{C
(kr^{k-1}+r^k)}{2^k\Gamma(k+\frac12)\Gamma(1/2)}\big(1+\frac1{k+1/2}\big),
\end{equation}
where $C$ is a constant and these estimates will be used when
$r\lesssim1$. Another well known asymptotic expansion about the
Bessel function is
\begin{equation*}
J_k(r)=r^{-1/2}\sqrt{\frac2{\pi}}\cos(r-\frac{k\pi}2-\frac{\pi}4)+O_{k}(r^{-3/2}),\quad
\text{as}~ r\rightarrow\infty,
\end{equation*}
but with a constant depending on $k$ (see \cite{SW}). As pointed out
in \cite{Stein1}, if one seeks a uniform bound for large $r$ and
$k$, then the best one can do is $|J_k(r)|\leq C r^{-\frac13}$. To
investigate the behavior of asymptotic on $k$ and $r$, we are
devoted to Schl\"afli's integral representation \cite{Watson} of the
Bessel function: for $r\in\R^+$ and $k>-\frac12$,
\begin{equation}\label{2.6}
\begin{split}
J_k(r)&=\frac1{2\pi}\int_{-\pi}^\pi
e^{ir\sin\theta-ik\theta}\mathrm{d}\theta-\frac{\sin(k\pi)}{\pi}\int_0^\infty
e^{-(r\sinh s+ks)}\mathrm{d}s\\&:=\tilde{J}_k(r)-E_k(r).
\end{split}
\end{equation}
We remark that $E_k(r)=0$ for $k\in\Z^+$. One easily estimates for
$r>0$
\begin{equation}\label{2.7}
|E_k(r)|=\Big|\frac{\sin(k\pi)}{\pi}\int_0^\infty e^{-(r\sinh
s+ks)}\mathrm{d}s\Big|\leq C (r+k)^{-1}.
\end{equation}

Next, we recall the properties of Bessel function $J_k(r)$ in
\cite{Stempak,Stein1}, the readers can also refer to \cite{MZZ1} for
the detailed proof.
\begin{lemma}[Asymptotics of the Bessel function] \label{Bessel} Assume $k\gg1$. Let $J_k(r)$ be
the Bessel function of order $k$ defined as above. Then there exist
a large constant $C$ and small constant $c$ independent of $k$ and
$r$ such that:

$\bullet$ when $r\leq \frac k2$
\begin{equation}\label{2.8}
\begin{split}
|J_k(r)|\leq C e^{-c(k+r)};
\end{split}
\end{equation}

$\bullet$ when $\frac k 2\leq r\leq 2k$
\begin{equation}\label{2.9}
\begin{split}
|J_k(r)|\leq C k^{-\frac13}(k^{-\frac13}|r-k|+1)^{-\frac14};
\end{split}
\end{equation}

$\bullet$ when $r\geq 2k$
\begin{equation}\label{2.10}
\begin{split}
 J_k(r)=r^{-\frac12}\sum_{\pm}a_\pm(r) e^{\pm ir}+E(r),
\end{split}
\end{equation}
where $|a_\pm(r)|\leq C$ and $|E(r)|\leq Cr^{-1}$.
\end{lemma}
For our purpose, we additionally need the asymptotic behavior of the
derivative of the Bessel function $J'_{k}(r)$. It is a
straightforward elaboration of the argument of proving Lemma
\ref{Bessel} in \cite{MZZ1}, but we give the proof for completeness.
\begin{lemma}\label{Bessel2} Assume $r, k\gg1$. Then there exists
a constant $C$ independent of $k$ and $r$ such that
$$|J'_k(r)|\leq C r^{-\frac12}.$$
\end{lemma}
\begin{proof}When $r\leq
\frac k2$ or $r\geq 2k$, we apply the recurrence formula
\cite{Watson}
$$J'_k(r)=\frac12\big(J_{k-1}(r)-J_{k+1}(r)\big),$$ \eqref{2.8} and
\eqref{2.10} to obtaining $|J'_k(r)|\leq C r^{-\frac12}$.

When $\frac k 2\leq r\leq 2k$, we have by \eqref{2.6}
\begin{equation*}
\begin{split}
J'_k(r)=\tilde{J}'_k(r)-E'_k(r).
\end{split}
\end{equation*}
A simple computation gives that for $r>0$
\begin{equation*}
|E'_k(r)|=\Big|\frac{\sin(k\pi)}{\pi}\int_0^\infty e^{-(r\sinh
s+ks)}\sinh s~\mathrm{d}s\Big|\leq C (r+k)^{-1}.
\end{equation*}
Thus we only need to estimate $\tilde{J}'_k(r)$. We divide two cases
$r>k$ and $r\leq k$ to estimate it by the stationary phase argument.
Let
$$\phi_{r,k}(\theta)=r\sin\theta-k\theta.$$

{\bf Case 1: $k<r\leq2k$.} Let $\theta_0=\cos^{-1}(\frac k r)$, then
$$\phi'_{r,k}(\theta_0)=r\cos\theta_0-k=0.$$

Now we split $\tilde{J}_k(r)$ into two pieces:
\begin{equation*}
\begin{split}
\tilde{J}'_k(r)&=\frac i{2\pi}\int_{\Omega_{\delta}}
e^{ir\sin\theta-ik\theta}\sin\theta~\mathrm{d}\theta+\frac
i{2\pi}\int_{B_{\delta}}
e^{ir\sin\theta-ik\theta}\sin\theta~\mathrm{d}\theta,
\end{split}
\end{equation*}
where
\begin{equation*}
\begin{split}
\Omega_{\delta}=\{\theta:|\theta\pm\theta_0|\leq \delta\},\quad
B_{\delta}=[-\pi,\pi]\setminus \Omega_{\delta}\quad \text{with}\quad
\delta>0.
\end{split}
\end{equation*}

We have by taking absolute values
\begin{equation*}
\begin{split}
\Big|\frac1{2\pi}\int_{\Omega_{\delta}}
e^{ir\sin\theta-ik\theta}\sin\theta~\mathrm{d}\theta\Big|\leq
C|\sin(\theta_0\pm\delta)|\delta.
\end{split}
\end{equation*}

Integrating by parts, we have
\begin{equation*}
\begin{split}
\int_{B_{\delta}}
e^{ir\sin\theta-ik\theta}\sin\theta~\mathrm{d}\theta=\frac{e^{i(r\sin\theta-k\theta)}\sin\theta}{i(r\cos\theta-k)}\Big|_{\partial
B_{\delta}}-\int_{B_{\delta}}\frac{
e^{ir\sin\theta-ik\theta}(r-k\cos\theta)}{i(r\cos\theta-k)^2}\mathrm{d}\theta,
\end{split}
\end{equation*}
where $\partial B_{\delta}=\{\pm\pi,\pm\theta_0\pm\delta\}$. It is
easy to see that
\begin{equation*}
\begin{split}
\Big|\frac{e^{i(r\sin\theta-k\theta)}\sin\theta}{i(r\cos\theta-k)}\Big|_{\partial
B_{\delta}}\Big|\leq
c\sin(\theta_0\pm\delta)|r\cos(\theta_0\pm\delta)-k|^{-1}.
\end{split}
\end{equation*}

Since $r-k\cos\theta>0$, we obtain
\begin{equation*}
\begin{split}
\Big|\int_{B_{\delta}}\frac{
e^{ir\sin\theta-ik\theta}(r-k\cos\theta)}{i(r\cos\theta-k)^2}\mathrm{d}\theta\Big|&\leq
\int_{B_{\delta}}\frac{
|r-k\cos\theta|}{(r\cos\theta-k)^2}\mathrm{d}\theta=
\frac{\sin\theta}{(r\cos\theta-k)}\Big|_{\partial B_{\delta}}\\&\leq
c|\sin(\theta_0\pm\delta)|\cdot|r\cos(\theta_0\pm\delta)-k|^{-1}.
\end{split}
\end{equation*}
Therefore,
$$|\tilde{J}'_k(r)|\leq
C|\sin(\theta_0\pm\delta)|\delta+c|\sin(\theta_0\pm\delta)|\cdot|r\cos(\theta_0\pm\delta)-k|^{-1}.$$
We shall choose proper $\delta$ such that
$$|\sin(\theta_0\pm\delta)|\delta\sim
c|\sin(\theta_0\pm\delta)|\cdot|r\cos(\theta_0\pm\delta)-k|^{-1}.$$
Noting that
$\cos(\theta_0\pm\delta)=\cos\theta_0\cos\delta\mp\sin\theta_0\sin\delta$
and the definition of $\theta_0$, we get
\begin{equation*}
\begin{split}
r\cos(\theta_0\pm\delta)-k=k\cos\delta\pm\sqrt{r^2-k^2}\sin\delta-k.
\end{split}
\end{equation*}
Since $1-\cos\delta=2\sin^2\frac{\delta}2$, one has
$$|r\cos(\theta_0\pm\delta)-k|\sim
|k\delta^2\pm\delta\sqrt{r^2-k^2}|\quad\text{with small}~\delta.$$
On the other hand, we have by
$\sin(\theta_0\pm\delta)=\sin\theta_0\cos\delta\pm\cos\theta_0\sin\delta$,
\begin{equation*}
\begin{split}
\sin(\theta_0\pm\delta)=\pm\frac{\sqrt{r^2-k^2}}r(1-\frac{\delta^2}2)\pm\frac
k r\delta.
\end{split}
\end{equation*}

 When $|r-k|\leq k^{\frac13}$, choosing $\delta=Ck^{-\frac13}$ with
large $C\geq2$, we have
\begin{equation*}
\begin{split}
|\sin(\theta_0\pm\delta)|\cdot|r\cos(\theta_0\pm\delta)-k|^{-1}\lesssim
k^{-\frac23}(C^2-Ck^{-\frac23}\sqrt{r^2-k^2})^{-1}\lesssim
k^{-\frac23}\lesssim r^{-\frac23}.
\end{split}
\end{equation*}

When $|r-k|\geq k^{\frac13}$, taking $\delta=c(r^2-k^2)^{-\frac14}$
with small $c>0$, we obtain
\begin{equation*}
\begin{split}
&|\sin(\theta_0\pm\delta)|\cdot|r\cos(\theta_0\pm\delta)-k|^{-1}\\\lesssim&
\big[|r-k|^{\frac12}r^{-\frac12}+r^{-1}+(r^2-k^2)^{-\frac14}\big]
(r^2-k^2)^{-\frac14}\big(c-c^2k(r^2-k^2)^{-\frac34}\big)^{-1}\\\lesssim&
k^{-\frac34}|r-k|^{\frac14}+k^{-\frac12}|r-k|^{-\frac12}\lesssim
r^{-\frac12},
\end{split}
\end{equation*}
where we use the fact that
$(r^2-k^2)^{-\frac14}\leq(2k)^{-\frac14}|r-k|^{-\frac14}$ for $k<r$.

{\bf Case 2: $\frac k2 \leq r\leq k $.} When $k-k^{\frac13}<r<k$,
choosing $\theta_0=0$ and $\delta=Ck^{-\frac13}$ with large
$C\geq2$, it follows from the above argument that
\begin{equation*}
\begin{split}
|\tilde{J}'_k(r)|&\lesssim|\sin(\theta_0\pm\delta)|\cdot|r\cos(\theta_0\pm\delta)-k|^{-1}\\&\lesssim
\delta({r\delta^2}/2-|r-k|)^{-1}\lesssim k^{-\frac23}\lesssim
r^{-\frac23}.
\end{split}
\end{equation*}

When $r<k-k^{\frac13}$, there is no critical point. Hence we obtain
$$|\tilde{J}'_k(r)|\lesssim|r-k|^{-2}\lesssim r^{-\frac23}.$$

Finally, we collect all the estimates to get
$|\tilde{J}'_k(r)|\lesssim r^{-\frac12}.$ \vspace{0.2cm}
\end{proof}

Next, we record the two basic results about the modified square
function expressions.
\begin{lemma}[A modified Littlewood-Paley theorem \cite{Stein1}]\label{square-expression} Let $\beta\in C_0^\infty(\R^+)$ be supported in
$[\frac12,2]$, $\beta_j(\rho)=\beta(2^{-j}\rho)$ and
$\sum\limits_{j}\beta_j=1$. Then for any $\nu(k)>0$ and
$1<p<\infty$, we have
\begin{equation}\label{2.11}
\begin{split}
&\bigg\|\sum_{j\in
\Z}\int_0^\infty(r\rho)^{-\frac{n-2}2}J_{\nu(k)}(r\rho)\cos(t\rho)b^0_{k,\ell}(\rho)\rho^{n-1}\beta_j(\rho)\mathrm{d}\rho\bigg\|_{L^p_{r^{n-1}}}\\&\sim
\bigg\|\Big(\sum_{j\in
\Z}\Big|\int_0^\infty(r\rho)^{-\frac{n-2}2}J_{\nu(k)}(r\rho)\cos(t\rho)
b^0_{k,\ell}(\rho)\rho^{n-1}\beta_j(\rho)\mathrm{d}\rho\Big|^2\Big)^{\frac12}\bigg\|_{L^p_{r^{n-1}}}.
\end{split}
\end{equation}
\end{lemma}

For the sake of the completeness, we will prove Lemma
\ref{square-expression} in the appendix by using a weak-type $(1,1)$
estimate for the multiplier operators with respect to the Hankel
transform.

\begin{lemma}[Littlewood-Paley-Stein theorem for the sphere,
\cite{Stein2,Stri,Sterbenz}]\label{square-expression2} Let $\beta\in
C_0^\infty(\R^+)$ be supported in $[\frac12,4]$ and $\beta(\rho)=1$
when $\rho\in[1,2]$. Assume $\beta_j(\rho)=\beta(2^{-j}\rho).$ Then
for any $1<p<\infty$ and any test function $f(\theta)$ defined on
$\mathbb{S}^{n-1}$, we have
\begin{equation}\label{2.12}
\|f(\theta)\|_{L^p_{\theta}(\mathbb{S}^{n-1})}\sim
\Big\|\Big(\big|a_{0,1}Y_{0,1}(\theta)\big|^2+\sum_{j=0}^\infty\big|\sum_{k}\sum_{\ell=1}^{{d}(k)}\beta_j(k)a_{k,\ell}Y_{k,\ell}(\theta)\big|^2\Big)^{\frac12}\Big\|_{L^p_{\theta}(\mathbb{S}^{n-1})},
\end{equation}
where
$f=\sum\limits_{k=0}^{\infty}\sum\limits_{\ell=1}^{{d}(k)}a_{k,\ell}Y_{k,\ell}(\theta)$.
\end{lemma}

We conclude this subsection by showing the ``Bernstein" inequality
on sphere
\begin{equation}\label{Bernstein}
\big\|\sum_{k=2^{j}}^{2^{j+1}}\sum_{\ell=1}^{{d}(k)} a_{k,\ell}
Y_{k,\ell}(\theta)\big\|_{L^q(\mathbb{S}^{n-1})}\leq C_{q,n}
2^{j(n-1)(\frac12-\frac1q)}\Big(\sum_{k=2^{j}}^{2^{j+1}}\sum_{\ell=1}^{{d}(k)}
|a_{k,\ell}|^2\Big)^{\frac12}
\end{equation}
for $q\geq2, j=0,1,2\cdots$.

In fact, since $\sum\limits_{\ell=1}^{{d}(k)}
|Y_{k,\ell}(\theta)|^2=d(k)|\mathbb{S}^{n-1}|^{-1}, \forall
\theta\in \mathbb{S}^{n-1}$ (see Stein-Weiss \cite{SW}), one has
\begin{equation*}
\begin{split}
\Big\|\sum_{k=2^{j}}^{2^{j+1}}\sum_{\ell=1}^{{d}(k)} a_{k,\ell}
Y_{k,\ell}(\theta)\Big\|^2_{L^\infty(\mathbb{S}^{n-1})}&\leq C
\sum_{k=2^{j}}^{2^{j+1}}\sum_{\ell=1}^{{d}(k)}
|a_{k,\ell}|^2\Big\|\Big(\sum_{k=2^{j}}^{2^{j+1}}\sum_{\ell=1}^{{d}(k)}
|Y_{k,\ell}(\theta)|^2\Big)^{\frac12}\Big\|^2_{L^\infty(\mathbb{S}^{n-1})}\\&\leq
C \sum_{k=2^{j}}^{2^{j+1}}\sum_{\ell=1}^{{d}(k)}
|a_{k,\ell}|^2\sum_{k=2^{j}}^{2^{j+1}}k^{n-2}\\&\leq C2^{j(n-1)}
\sum_{k=2^{j}}^{2^{j+1}}\sum_{\ell=1}^{{d}(k)} |a_{k,\ell}|^2.
\end{split}
\end{equation*}
Interpolating this with
\begin{equation*}
\Big\|\sum_{k=2^{j}}^{2^{j+1}}\sum_{\ell=1}^{{d}(k)} a_{k,\ell}
Y_{k,\ell}(\theta)\Big\|^2_{L^2(\mathbb{S}^{n-1})}\leq C
\sum_{k=2^{j}}^{2^{j+1}}\sum_{\ell=1}^{{d}(k)} |a_{k,\ell}|^2
\end{equation*}
yields \eqref{Bernstein}.

\subsection{Spectrum of $-\Delta+\frac{a}{|x|^2}$ and Hankel transform}
Let us first consider the eigenvalue problem associated with the
operator $-\Delta+\frac{a}{|x|^2}$:
\begin{equation*}
\begin{cases}
-\Delta u+\frac{a}{|x|^2}u=\rho^2 u\quad x\in B=\{x:|x|\leq1\},\\
u(x)=0,\qquad x\in \mathbb{S}^{n-1}.
\end{cases}
\end{equation*}

If $u(x)=f(r)Y_k(\theta)$, we have
\begin{equation*}
f''(r)+\frac{n-1}r f'(r)+[\rho^2-\frac{k(k+n-2)+a}{r^2}]f(r)=0.
\end{equation*}

Let $\lambda=\rho r$ and $f(r)=\lambda^{-\frac{n-2}2}g(\lambda)$, we
obtain
\begin{equation}\label{bessfunc}
g''(\lambda)+\frac{1}\lambda
g'(\lambda)+[1-\frac{(k+\frac{n-2}2)^2+a}{\lambda^2}]g(\lambda)=0.
\end{equation}

Define
\begin{equation}\label{2.14}
\mu(k)=\frac{n-2}2+k,\quad\text{and}\quad\nu(k)=\sqrt{\mu^2(k)+a}\quad\text{with}\quad
a>-(n-2)^2/4.
\end{equation}
The Bessel function $J_{\nu(k)}(\lambda)$ solves the Bessel equation
\eqref{bessfunc}. And the eigenfunctions corresponding to the
spectrum $\rho^2$ can be  expressed by
\begin{equation}\label{2.16}
\phi_{\rho}(x)=(\rho r)^{-\frac{n-2}2}J_{\nu(k)}(\rho
r)Y_k(\theta)\quad\text{with}\quad x=r\theta,
\end{equation}
where
\begin{equation}\label{2.15}
\Big(-\Delta+\frac{a}{|x|^2}\Big)\phi_{\rho}=\rho^2\phi_{\rho}.
\end{equation}

We define the following elliptic operator
\begin{equation}\label{2.17}
\begin{split}
A_{\nu(k)}:&=-\partial_r^2-\frac{n-1}r\partial_r+\frac{k(k+n-2)+a}{r^2}\\&=-\partial_r^2-\frac{n-1}r\partial_r+\frac{\nu^2(k)-\big(\frac{n-2}2\big)^2}{r^2},
\end{split}
\end{equation}
then $A_{\nu(k)}\phi_{\rho}=\rho^2\phi_{\rho}$.  Define the Hankel
transform of order $\nu$:
\begin{equation}\label{2.18}
(\mathcal{H}_{\nu}f)(\xi)=\int_0^\infty(r\rho)^{-\frac{n-2}2}J_{\nu}(r\rho)f(r\omega)r^{n-1}\mathrm{d}r,
\end{equation}
where $\rho=|\xi|$, $\omega=\xi/|\xi|$ and $J_{\nu}$ is the Bessel
function of order $\nu$. Specially, if the function $f$ is radial,
then
\begin{equation}\label{2.19}
(\mathcal{H}_{\nu}f)(\rho)=\int_0^\infty(r\rho)^{-\frac{n-2}2}J_{\nu}(r\rho)f(r)r^{n-1}\mathrm{d}r.
\end{equation}

If $
f(x)=\sum\limits_{k=0}^{\infty}\sum\limits_{\ell=1}^{d(k)}a_{k,\ell}(r)Y_{k,\ell}(\theta)
$, then we obtain by \eqref{2.3}
\begin{equation}\label{2.20}
\begin{split}
\hat f(\xi)=\sum_{k=0}^{\infty}\sum_{\ell=1}^{d(k)}2\pi
i^{k}Y_{k,\ell}(\omega)\big(\mathcal{H}_{\mu(k)}a_{k,\ell}\big)(\rho).
\end{split}
\end{equation}

We will also make use of the following properties of the Hankel
transform, which appears in  \cite{BPSS,PSS}.
\begin{lemma}\label{Hankel3}
Let $\mathcal{H}_{\nu}$ and $A_{\nu}$ be defined as above. Then

$(\rm{i})$ $\mathcal{H}_{\nu}=\mathcal{H}_{\nu}^{-1}$,

$(\rm{ii})$ $\mathcal{H}_{\nu}$ is self-adjoint, i.e.
$\mathcal{H}_{\nu}=\mathcal{H}_{\nu}^*$,

$(\rm{iii})$ $\mathcal{H}_{\nu}$ is an $L^2$ isometry, i.e.
$\|\mathcal{H}_{\nu}\phi\|_{L^2_\xi}=\|\phi\|_{L^2_x}$,

$(\rm{iv})$ $\mathcal{H}_{\nu}(
A_{\nu}\phi)(\xi)=|\xi|^2(\mathcal{H}_{\nu} \phi)(\xi)$, for
$\phi\in L^2$.
\end{lemma}
Let $\mathcal{K}_{\mu,\nu}^0=\mathcal{H}_{\mu}\mathcal{H}_{\nu}$,
then as well as in \cite{PSS} one has
\begin{equation}\label{2.21}
A_{\mu}\mathcal{K}_{\mu,\nu}^0=\mathcal{K}_{\mu,\nu}^0A_{\nu}.
\end{equation}
For our purpose, we need another crucial properties of
$\mathcal{K}_{\mu(k),\nu(k)}^0$ with $k=0$:
\begin{lemma}[The boundness of $\mathcal{K}_{\lambda,\nu}^0$, \cite{BPSS,PSS}]\label{continuous}
Let $\nu, \alpha, \beta\in \R$, $\nu>-1,
\lambda=\mu(0)=\frac{n-2}2$, $-n<\alpha<2(\nu+1)$ and
$-2(\nu+1)<\beta<n$. Then the conjugation operator
$\mathcal{K}_{\lambda,\nu}^0$ is continuous on $\dot
H^{\beta}_{p,\text{rad}}(\R^n)$ provided that
\begin{equation*}
\max\Big\{0,\frac{\lambda-\nu}{n},\frac\beta
n\Big\}<\frac1p<\min\Big\{\frac{\lambda+\nu+2}{n},\frac{\lambda+\nu+2+\beta}{n},1\Big\}
\end{equation*}
while the inverse operator $\mathcal{K}_{\nu,\lambda}^0$ is
continuous on $\dot H^{\alpha}_{q,\text{rad}}(\R^n)$ provided that
\begin{equation*}
\max\Big\{0,\frac{\lambda-\nu}{n},\frac{\lambda-\nu+\alpha}
n\Big\}<\frac1q<\min\Big\{\frac{\lambda+\nu+2}{n},1+\frac{\alpha}{n},1\Big\}.
\end{equation*}
\end{lemma}
We also need the Strichartz estimates for \eqref{1.1} in
\cite{BPSS}:
\begin{lemma}[Strichartz estimates]\label{stri1}
For $n\geq2$, let $2\leq r<\infty$ and $q,r,\gamma,\sigma$ satisfy
\begin{equation}\label{2.22}
\frac1q\leq\min\Big\{\frac12,\frac{n-1}2\Big(\frac12-\frac1r\Big)\Big\},\quad
\sigma=\gamma+\frac1q-n\Big(\frac12-\frac1r\Big).
\end{equation} There exists a positive constant $C$ depending on $n, a, q,r,\gamma$ such that the solution $u$ of \eqref{1.1} satisfies
\begin{equation}\label{2.23}
\big\|(-\Delta)^{\frac\sigma2}u\big\|_{L^q_t(\R;L^r(\R^n))}\leq
C\big(\|u_0\|_{\dot H^{\gamma}}+\|u_1\|_{\dot H^{\gamma-1}}\big)
\end{equation}
provided that when $n=2,3$
\begin{equation*}
-\min\Big\{\frac{n-1}2,\nu(1)-\frac12,1+\nu(0)\Big\}<\gamma<\min\Big\{\frac{n+1}2,\nu(1)+\frac12,1+\nu(0)-\frac1q\Big\}
\end{equation*}
and when $n\geq4$
\begin{equation*}
-\min\Big\{\frac{n}2-\frac{n+3}{2(n-1)},\nu(1)-\frac{n+3}{2(n-1)},1+\nu(0)\Big\}<\gamma<\min\Big\{\frac{n+1}2,\nu(1)+\frac12,1+\nu(0)-\frac1q\Big\}.
\end{equation*}
\end{lemma}
Next, define the projectors $M_{jj'}=P_j\tilde{P}_{j'}$ and
$N_{jj'}=\tilde{P}_jP_{j'}$, where $P_j$ is the usual dyadic
frequency localization at $|\xi|\sim 2^{j}$ and $\tilde{P}_j$ is the
localization with respect to $\big(-\Delta+\frac
a{|x|^2}\big)^{\frac12}$. More precisely, let $f$ be in  the $k$-th
harmonic subspace, then
\begin{equation*}
P_j f=\mathcal{H}_{\mu(k)}\beta_j
\mathcal{H}_{\mu(k)}f\quad\text{and}\quad \tilde{P}_j
f=\mathcal{H}_{\nu(k)}\beta_j \mathcal{H}_{\nu(k)}f,
\end{equation*}
where $\beta_j(\xi)=\beta(2^{-j}|\xi|)$ with $\beta\in
C_0^\infty(\R^+)$ supported in $[\frac14,2]$.  Then, we have the
almost orthogonality estimate which is proved in \cite{BPSS}.
\begin{lemma}[Almost orthogonality estimate, \cite{BPSS}]\label{orthogonality}
There exists a positive constant $C$ independent of $j,j',$ and $k$
such that the following inequalities hold for all positive
$\epsilon_1<1+\min\{\frac{n-2}2, (\frac{(n-2)^2}4+a)^{\frac12}\}$
$$\|M_{j j'}f\|_{L^2(\R^n)},~~ \|N_{j j'}f\|_{L^2(\R^n)}\leq C
2^{-\epsilon_1|j-j'|}\|f\|_{L^2(\R^n)}, $$ where $f$ is in  the
$k$-th harmonic subspace.
\end{lemma}

As a consequence of Lemma \ref{stri1} and Lemma \ref{orthogonality},
we have
\begin{lemma}[Strichartz estimates for unit frequency]\label{stri}
Let $n\geq3$, $k\in \N$. Let $u$ solve
\begin{equation*}
\begin{cases}
(\partial_t^2-\Delta+\frac a{|x|^2})u=0, \\
u|_{t=0}=u_0(x),~u_t|_{t=0}=0,
\end{cases}
\end{equation*} where $u_0\in L^2(\R^n)$ and
$$u_0=\sum\limits_{k=0}^\infty\sum\limits_{\ell=1}^{{d}(k)}a_{k,\ell}(r)Y_{k,\ell}(\theta).$$
Assume that for all $k,\ell\in\N$,
$\text{supp}~\big[\mathcal{H}_{\nu(k)}a_{k,\ell}\big]\subset [1,2]$.
Then the following estimate holds for
$a>\frac1{(n-1)^2}-\frac{(n-2)^2}4$
\begin{equation}\label{2.24}
\|u(t,x)\|_{L^q(\R;L^r(\R^n))}\leq C\|u_0\|_{L^2(\R^n)},
\end{equation}
where $q\geq2,$ $\frac1q=\frac{n-1}2(\frac12-\frac1r)$ and $
(q,r,n)\neq(2,\infty,3)$.
\end{lemma}
\begin{proof}
By making use of Lemma \ref{stri1} with
$\sigma=0,\gamma=n(\frac12-\frac1r)-\frac1q=\frac{n+1}{q(n-1)}$, we
obtain that $\|u\|_{L^q_t(\R;L^r(\R^n))}\leq C\|u_0\|_{\dot
H^{\gamma}}$. Since $0<\gamma\leq1$,  we have by Lemma
\ref{orthogonality} with $\epsilon_1=1+$,
\begin{equation*}
\begin{split}\|u\|_{L^q_t(\R;L^r(\R^n))}&\leq C\Big(\sum_{j\in\Z}
2^{2j\gamma}\|P_j u_0\|^2_{ L^{2}}\Big)^{\frac12}=
C\Big(\sum_{j\in\Z} 2^{2j\gamma}\Big\|
\sum_{k=0}^\infty\sum_{\ell=1}^{{d}(k)}P_j\big(a_{k,\ell}(r)Y_{k,\ell}(\theta)\big)\Big\|^2_{
L^{2}}\Big)^{\frac12}\\&=C\Big(\sum_{j\in\Z} 2^{2j\gamma}\Big\|
\sum_{k=0}^\infty\sum_{\ell=1}^{{d}(k)}P_j\tilde{P}_1\big(a_{k,\ell}(r)Y_{k,\ell}(\theta)\big)\Big\|^2_{
L^{2}}\Big)^{\frac12} \\&\leq C\Big(\sum_{j\in\Z}
2^{2j\gamma-2\epsilon_1|j-1|}\|u_0\|^2_{ L^{2}}\Big)^{\frac12}\leq
C\|u_0\|_{ L^{2}}.
\end{split}
\end{equation*}
This completes the proof of Lemma \ref{stri}.
\end{proof}

\section{Estimates of Hankel transforms}
In this section, we prove some estimates for the Hankel transforms
of order $\nu(k)$. These estimates will be utilized to prove the
main results in the next section.
\begin{proposition}\label{Hankel1}
Let  $k\in \N, 1\leq\ell\leq d(k)$ and let $\varphi$ be a smooth
function supported in the interval $I:=[\frac12,2]$. Then
\begin{equation}\label{3.1}
\begin{split}
\Big\|\int_0^\infty e^{- it\rho} &{J}_{\nu(k)}( r\rho)
b^0_{k,\ell}(\rho) \varphi(\rho)\mathrm{d}\rho
\Big\|_{L^2_t(\R;L^2_{r}([R,2R]))}\leq
C\min\big\{R^{\frac12},1\big\} \|b^0_{k,\ell}(\rho)\|_{L^2_\rho(I)},
\end{split}
\end{equation}
where $R\in 2^{\Z}$ and $C$ is a constant independent of $R, k,$ and
$\ell$.
\end{proposition}
\begin{proof}
Using the Plancherel theorem in $t$, we have
\begin{equation}\label{3.2}
\begin{split}
\text{L.H.S of }~\eqref{3.1}\lesssim \Big\| \big\|J_{\nu(k)}( r\rho)
b^0_{k,\ell}(\rho) \varphi(\rho)\big\|_{L^{2}_\rho}
\Big\|_{L^2_{r}([R,2R])}.
\end{split}
\end{equation}

We first consider the case $R\lesssim1$. Since $\nu(k)>0$, one has
by \eqref{2.4}
\begin{equation}\label{3.3}
\begin{split}
\text{L.H.S of}~\eqref{3.1}&\lesssim \big\|
b^0_{k,\ell}(\rho)\big\|_{L^{2}_\rho(I)}\Big(\int_R^{2R} \Big|\frac{
r^{\nu(k)}}{2^{\nu(k)}\Gamma(\nu(k)+\frac12)\Gamma(\frac12)}\Big|^{2}
\mathrm{d}r\Big)^{\frac12}\\&\lesssim R^{\frac
12}\big\|b^0_{k,\ell}(\rho)\big\|_{L^2_\rho(I)}.
\end{split}
\end{equation}

Next we consider the case $R\gg 1$. It follows from \eqref{3.2} that
\eqref{3.1} can be reduced to show
\begin{equation}\label{3.4}
\int_R^{2R}|J_{k}(r)|^2\mathrm{d}r\leq C,\quad R\gg 1,
\end{equation}
where the constant $C$ is independent of $k$ and $R$. To prove
\eqref{3.4}, we write
\begin{equation*}
\begin{split}
\int_R^{2R}|J_{k}(r)|^2\mathrm{d}r=\int_{I_1}|J_{k}(r)|^2\mathrm{d}r
+\int_{I_2}|J_{k}(r)|^2\mathrm{d}r+\int_{I_3}|J_{k}(r)|^2\mathrm{d}r
\end{split}
\end{equation*}
where $$I_1=[R,2R]\cap[0,\frac k 2],\quad I_2=[R,2R]\cap[\frac k
2,2k]\quad \text{and}\quad I_3=[R,2R]\cap[2k,\infty].$$

Using \eqref{2.8} and \eqref{2.10} in Lemma \ref{Bessel}, we have
\begin{equation}\label{3.5}
\begin{split}
\int_{I_1}|J_{k}(r)|^2\mathrm{d}r\leq C
\int_{I_1}e^{-cr}\mathrm{d}r\leq C e^{-cR},
\end{split}
\end{equation}
and \begin{equation}\label{3.6}
\begin{split}\int_{I_3}|J_{k}(r)|^2\mathrm{d}r\leq C.\end{split}
\end{equation}
On the other hand, one has by \eqref{2.9}
\begin{equation*}
\begin{split}
\int_{[\frac k 2,2k]}|J_{k}(r)|^2\mathrm{d}r&\leq C \int_{[\frac k
2,2k]}k^{-\frac23}(1+k^{-\frac13}|r-k|)^{-\frac 12}\mathrm{d}r\leq
C.
\end{split}
\end{equation*}
Observing $[R,2R]\cap[\frac k 2,2k]=\emptyset$ unless $R\sim k$, we
obtain
\begin{equation}\label{3.7}
\begin{split}
\int_{I_2}|J_{k}(r)|^2\mathrm{d}r\leq C.
\end{split}
\end{equation}
This together with \eqref{3.5} and \eqref{3.6} yields \eqref{3.1}.
\end{proof}

\begin{proposition}\label{Hankel2}Let $\gamma\geq2$ and let $k\in \N, 1\leq\ell\leq d(k)$.
Suppose $\text{supp}~b^0_{k,\ell}(\rho)\subset I:=[1,2]$.  Then
\begin{equation}\label{3.8}
\begin{split}
\Big\|\mathcal{H}_{\nu(k)}&\big[\cos(
t\rho)b^0_{k,\ell}(\rho)\big](r)\Big\|_{L^2_t(\R;L^\gamma_{r^{n-1}\mathrm{d}r}([R,2R]))}
\\&\leq C \min\Big\{R^{\frac{(n+1)+(\gamma-2)\nu(k)}\gamma-\frac{n-1}
2}, R^{\frac{n-1}\gamma-\frac{n-2}2}\Big\}
\|b^0_{k,\ell}(\rho)\|_{L^2_\rho(I)},
\end{split}
\end{equation}
where $R\in 2^{\Z}$ and $C$ is a constant independent of $R, k$ and
$\ell$.
\end{proposition}

\begin{proof}
We first consider the case $R\gg 1$. Using the definition of Hankel
transform and the interpolation, we only need to prove
\begin{equation}\label{3.9}
\begin{split}
\Big\|\int_0^\infty e^{- it\rho} {J}_{\nu(k)}(
r\rho)b^0_{k,\ell}(\rho)
(r\rho)^{-\frac{n-2}2}&\rho^{n-1}\mathrm{d}\rho\Big\|_{L^2_t(\R;L^2_{r^{n-1}\mathrm{d}r}([R,2R]))}\\&
\lesssim R^{\frac1 2}\| b^0_{k,\ell}(\rho)\|_{L^2_\rho},
\end{split}
\end{equation}
and
\begin{equation}\label{3.10}
\begin{split}
\Big\|\int_0^\infty e^{- it\rho} {J}_{\nu(k)}(
r\rho)b^0_{k,\ell}(\rho)
(r\rho)^{-\frac{n-2}2}&\rho^{n-1}\mathrm{d}\rho\Big\|_{L^2_t(\R;L^\infty_{r^{n-1}\mathrm{d}r}([R,2R]))}\\&
\lesssim R^{-\frac{n-2}2}\| b^0_{k,\ell}(\rho)\|_{L^2_\rho}.
\end{split}
\end{equation}
\eqref{3.9} follows from Proposition \ref{Hankel1}. To prove
\eqref{3.10}, it is enough to show that there exists a constant $C$
independent of $k,\ell$ such that
\begin{equation}\label{3.11}
\begin{split}
\Big\|\int_0^\infty e^{-it\rho} J_{\nu(k)}( r\rho)
b^0_{k,\ell}(\rho) \varphi(\rho)\mathrm{d}\rho
\Big\|_{L^2_t(\R;L^\infty_{r}([R,2R]))}\leq C
\|b^0_{k,\ell}(\rho)\|_{L^2_\rho(I)}.
\end{split}
\end{equation}

By the Sobolev embedding $H^1(\Omega)\hookrightarrow
L^\infty(\Omega)$ with $\Omega=[R,2R]$, it suffices to show
\begin{equation}\label{3.12}
\begin{split}
\Big\|\int_0^\infty e^{- it\rho} J_{\nu(k)}(r\rho)
b^0_{k,\ell}(\rho) &\varphi(\rho)\mathrm{d}\rho
\Big\|_{L^2_t(\R;L^2_{r}([R,2R]))}\leq C
\|b^0_{k,\ell}(\rho)\|_{L^2_\rho(I)},
\end{split}
\end{equation}
and
\begin{equation}\label{3.13}
\begin{split}
\Big\|\int_0^\infty e^{- it\rho}  J^{\prime}_{\nu(k)}(r\rho)
b^0_{k,\ell}(\rho) \rho \varphi(\rho)\mathrm{d}\rho
\Big\|_{L^2_t(\R;L^2_{r}([R,2R]))}\leq C
\|b^0_{k,\ell}(\rho)\|_{L^2_\rho(I)}.
\end{split}
\end{equation}
In fact,  \eqref{3.12} follows from Proposition \ref{Hankel1}, we
apply the Plancherel theorem in $t$ and Lemma 2.2 to showing
\eqref{3.13}.

Secondly, we consider the case $R\lesssim1$. From the definition of
Hankel transform, we need to prove
\begin{equation}\label{3.14}
\begin{split}
\Big\|\int_0^\infty e^{- it\rho} {J}_{\nu(k)}(
r\rho)&b^0_{k,\ell}(\rho)
(r\rho)^{-\frac{n-2}2}\rho^{n-1}\mathrm{d}\rho\Big\|_{L^2_t(\R;L^\gamma_{r}([R,2R]))}
\\&\lesssim R^{\frac{2+(\gamma-2)\nu(k)}\gamma-\frac{n-1} 2}\|
b^0_{k,\ell}(\rho)\|_{L^2_\rho}.
\end{split}
\end{equation}
On the other hand, we have by Proposition \ref{Hankel1}
\begin{equation}\label{3.15}
\begin{split}
\Big\|\int_0^\infty e^{-it\rho} {J}_{\nu(k)}(
r\rho)b^0_{k,\ell}(\rho)
&(r\rho)^{-\frac{n-2}2}\rho^{n-1}\mathrm{d}\rho\Big\|_{L^2_t(\R;L^2_{r}([R,2R]))}\\&
\lesssim R^{-\frac{n-3} 2}\| b^0_{k,\ell}(\rho)\|_{L^2_\rho}.
\end{split}
\end{equation}

By interpolation, it suffices to prove the estimate
\begin{equation}\label{3.16}
\begin{split}
\Big\|\int_0^\infty e^{- it\rho} {J}_{\nu(k)}(
r\rho)b^0_{k,\ell}(\rho)
&(r\rho)^{-\frac{n-2}2}\rho^{n-1}\mathrm{d}\rho\Big\|_{L^2_t(\R;L^\infty_{r}([R,2R]))}\\&
\lesssim R^{-\frac{n-1} 2+\nu(k)}\| b^0_{k,\ell}(\rho)\|_{L^2_\rho}.
\end{split}
\end{equation}
Indeed, using Sobolev embedding, we can prove \eqref{3.16} by
showing
\begin{equation*}
\begin{split}
\Big\|\int_0^\infty e^{- it\rho} J_{\nu(k)}(r\rho)
b^0_{k,\ell}(\rho) &\varphi(\rho)\mathrm{d}\rho
\Big\|_{L^2_t(\R;L^2_{r}([R,2R]))}\leq C R^{\frac12+\nu(k)}
\|b^0_{k,\ell}(\rho)\|_{L^2_\rho(I)},
\end{split}
\end{equation*}
and
\begin{equation*}
\begin{split}
\Big\|\int_0^\infty e^{- it\rho}  J^{\prime}_{\nu(k)}( r\rho)
b^0_{k,\ell}(\rho) \rho \varphi(\rho)\mathrm{d}\rho
\Big\|_{L^2_t(\R;L^2_{r}([R,2R]))}\leq C R^{\nu(k)-\frac12}
\|b^0_{k,\ell}(\rho)\|_{L^2_\rho(I)}.
\end{split}
\end{equation*}
These two estimates are implied by \eqref{2.4} and \eqref{2.5}.
Therefore, we conclude this proposition.

\end{proof}

\section{Proof of Theorem \ref{thm}}
In this section, we use Proposition \ref{Hankel1} and Proposition
\ref{Hankel2} to prove Theorem \ref{thm}. We first consider the
Cauchy problem:
\begin{equation}\label{4.1}
\begin{cases}
(\partial_{tt}-\Delta +\frac{a}{|x|^2})u(x,t)=0,\\
u(x,0)=u_0(x),~\partial_tu(x,0)=0.
\end{cases}
\end{equation}
We use the spherical harmonic expansion to write
\begin{equation}\label{4.2}
u_0(x)=\sum_{k=0}^{\infty}\sum_{\ell=1}^{d(k)}a^0_{k,\ell}(r)Y_{k,\ell}(\theta).
\end{equation}
Then we have the following proposition:
\begin{proposition}\label{pro}
Let $\gamma=\frac{2(n-1)}{n-2}+$ and suppose
$\text{supp}~\big(\mathcal{H}_{\nu}a^0_{k,\ell}\big)\subset [1,2]$
for all $k,\ell\in\N$ and $1\leq\ell\leq d(k)$. Then
\begin{equation}\label{4.3}
\begin{split}
\|u(x,t)\|_{L^2_tL^\gamma_{r^{n-1}\mathrm{d}r}L^2(\mathbb{S}^{n-1})}\leq
C\|u_0\|_{L^2_x}.
\end{split}
\end{equation}
\end{proposition}
\begin{proof}
Let us consider the equation \eqref{4.1} in polar coordinates. Write
$v(t,r,\theta)=u(t,r\theta)$ and $g(r,\theta)=u_0(r\theta)$. Then
$v(t,r,\theta)$ satisfies that
\begin{equation}\label{4.4}
\begin{cases}
\partial_{tt}
v-\partial_{rr}v-\frac{n-1}r\partial_rv-\frac1{r^2}\Delta_{\theta}v+\frac{a}{r^2}v=0,\\
v(0,r,\theta)=g(r,\theta),\quad\partial_t v(0,r,\theta)=0.
\end{cases}
\end{equation}
By \eqref{4.2}, we also have
\begin{equation*}
g(r,\theta)=\sum_{k=0}^{\infty}\sum_{\ell=1}^{d(k)}a^0_{k,\ell}(r)Y_{k,\ell}(\theta).
\end{equation*}
Using separation of variables, we can write $v$ as a superposition
\begin{equation}\label{4.5}
v(t,r,\theta)=\sum_{k=0}^{\infty}\sum_{\ell=1}^{d(k)}v_{k,\ell}(t,r)Y_{k,\ell}(\theta),
\end{equation}
where $v_{k,\ell}$ satisfies the following equation
\begin{equation*}
\begin{cases}
\partial_{tt}
v_{k,\ell}-\partial_{rr}v_{k,\ell}-\frac{n-1}r\partial_rv_{k,\ell}+\frac{k(k+n-2)+a}{r^2}v_{k,\ell}=0 \\
v_{k,\ell}(0,r)=a^0_{k,\ell}(r),\qquad\partial_t v_{k,\ell}(0,r)=0
\end{cases}
\end{equation*}
for each $k,\ell\in \N,$ and $1\leq\ell\leq d(k)$. From the
definition of $A_{\nu}$, it becomes
\begin{equation}\label{4.6}
\begin{cases}
\partial_{tt}
v_{k,\ell}+A_{\nu(k)}v_{k,\ell}=0, \\
v_{k,\ell}(0,r)=a^0_{k,\ell}(r),\qquad\partial_t v_{k,\ell}(0,r)=0.
\end{cases}
\end{equation}
Applying the Hankel transform to the equation \eqref{4.6}, we have
by Lemma \ref{Hankel3}
\begin{equation}\label{4.7}
\begin{cases}
\partial_{tt}
\tilde{ v}_{k,\ell}+\rho^2\tilde{v}_{k,\ell}=0, \\
\tilde{v}_{k,\ell}(0,\rho)=b^0_{k,\ell}(\rho),\qquad\partial_t\tilde{v}_{k,\ell}(0,\rho)=0,
\end{cases}
\end{equation}
where
\begin{equation}\label{4.8}
\tilde{v}_{k,\ell}(t,\rho)=(\mathcal{H}_{\nu}
v_{k,\ell})(t,\rho),\quad
b^0_{k,\ell}(\rho)=(\mathcal{H}_{\nu}a^0_{k,\ell})(\rho).
\end{equation}
Solving this ODE and using the inverse Hankel transform, we obtain
\begin{equation*}
\begin{split}
v_{k,\ell}(t,r)&=\int_0^\infty(r\rho)^{-\frac{n-2}2}J_{\nu(k)}(r\rho)\tilde{v}_{k,\ell}(t,\rho)\rho^{n-1}\mathrm{d}\rho\\
&=\frac1{2}\int_0^\infty(r\rho)^{-\frac{n-2}2}J_{\nu(k)}(r\rho)\big(e^{
it\rho}+e^{-it\rho}\big)b^0_{k,\ell}(\rho)\rho^{n-1}\mathrm{d}\rho.
\end{split}
\end{equation*}
Therefore, we get
\begin{equation}\label{4.9}
\begin{split} &u(x,t)=v(t,r,\theta)\\&=\sum_{k=0}^{\infty}\sum_{\ell=1}^{d(k)}Y_{k,\ell}(\theta)\int_0^\infty(r\rho)^{-\frac{n-2}2}J_{\nu(k)}(r\rho)\cos(
t\rho)b^0_{k,\ell}(\rho)\rho^{n-1}\mathrm{d}\rho\\&=\sum_{k=0}^{\infty}\sum_{\ell=1}^{d(k)}Y_{k,\ell}(\theta)\mathcal{H}_{\nu(k)}\big[\cos(
t\rho)b^0_{k,\ell}(\rho)\big](r).
\end{split}
\end{equation}
To prove \eqref{4.3}, it suffices to show
\begin{equation}\label{4.10}
\begin{split}
\Big\|\Big(\sum_{k=0}^{\infty}\sum_{\ell=1}^{d(k)}\big|\mathcal{H}_{\nu(k)}\big[\cos(
t\rho)b^0_{k,\ell}(\rho)\big](r)\big|^2\Big)^{\frac12}
\Big\|_{L^2_t(\R;L^\gamma_{r^{n-1}\mathrm{d}r}(\R^+))}\leq
C\|u_0\|_{L^2_x}.
\end{split}
\end{equation}
Using the dyadic decomposition, we have by $\ell^{2}\hookrightarrow
\ell^{\gamma}(\gamma>2)$
\begin{equation}\label{4.11}
\begin{split}
&\Big\|\Big(\sum_{k=0}^{\infty}\sum_{\ell=1}^{d(k)}\big|\mathcal{H}_{\nu(k)}\big[\cos(
t\rho)b^0_{k,\ell}(\rho)\big](r)\big|^2\Big)^{\frac12}
\Big\|^2_{L^2_t(\R;L^\gamma_{r^{n-1}\mathrm{d}r}(\R^+))}\\&=
\Big\|\Big(\sum_{R\in2^{\Z}}\Big\|\Big(\sum_{k=0}^{\infty}\sum_{\ell=1}^{d(k)}\big|\mathcal{H}_{\nu(k)}\big[\cos(
t\rho)b^0_{k,\ell}(\rho)\big](r)\big|^2\Big)^{\frac12}\Big\|^\gamma_{L^\gamma_{r^{n-1}\mathrm{d}r}([R,2R])}\Big)^{\frac1\gamma}\Big\|^2_{L^2_t(\R)}
\\&\lesssim
\sum_{R\in2^{\Z}}\sum_{k=0}^{\infty}\sum_{\ell=1}^{d(k)}\Big\|\mathcal{H}_{\nu(k)}\big[\cos(
t\rho)b^0_{k,\ell}(\rho)\big](r)\Big\|^2_{L^2_t(\R;L^\gamma_{r^{n-1}\mathrm{d}r}([R,2R]))}.
\end{split}
\end{equation}
By Proposition \ref{Hankel2}, we obtain
\begin{equation}\label{4.12}
\begin{split}
&\Big\|\Big(\sum_{k=0}^{\infty}\sum_{\ell=1}^{d(k)}\big|\mathcal{H}_{\nu(k)}\big[\cos(
t\rho)b^0_{k,\ell}(\rho)\big](r)\big|^2\Big)^{\frac12}
\Big\|^2_{L^2_t(\R;L^\gamma_{r^{n-1}\mathrm{d}r}(\R^+))}
\\&\lesssim
\sum_{R\in2^{\Z}}\sum_{k=0}^{\infty}\sum_{\ell=1}^{d(k)}\min\Big\{R^{\frac{(n+1)+(\gamma-2)\nu(k)}\gamma-\frac{n-1}
2}, R^{\frac{n-1}\gamma-\frac{n-2}2}\Big\}^2\|
b^0_{k,\ell}(\rho)\|^2_{L^2_\rho}\\&\lesssim
\sum_{k=0}^{\infty}\sum_{\ell=1}^{d(k)}\|
b^0_{k,\ell}(\rho)\|^2_{L^2_\rho}.
\end{split}
\end{equation}
Since $\text{supp}~b^0_{k,\ell}(\rho)\subset[1,2]$, we have
\begin{equation*}
\begin{split}
\sum_{k=0}^{\infty}\sum_{\ell=1}^{d(k)}\|
b^0_{k,\ell}(\rho)\|^2_{L^2_\rho} \lesssim
\sum_{k=0}^{\infty}\sum_{\ell=1}^{d(k)}\|
\big(\mathcal{H}_{\nu(k)}a^0_{k,\ell}\big)(\rho)\|^2_{L^2_{\rho^{n-1}\mathrm{d}\rho}}.
\end{split}
\end{equation*}
It follows from Lemma \ref{Hankel3} that
\begin{equation*}
\begin{split}
\sum_{k=0}^{\infty}\sum_{\ell=1}^{d(k)}\|
\big(\mathcal{H}_{\nu(k)}a^0_{k,\ell}\big)(\rho)\|^2_{L^2_{\rho^{n-1}\mathrm{d}\rho}}=
\sum_{k=0}^{\infty}\sum_{\ell=1}^{d(k)}\|
a^0_{k,\ell}(r)\|^2_{L^2_{r^{n-1}\mathrm{d}r}}=\|u_0(x)\|^2_{L^2_x(\R^n)}.
\end{split}
\end{equation*}
Therefore, we complete the proof of \eqref{4.3}.

\end{proof}
Now we turn to prove Theorem \ref{thm}. We choose $\beta\in
C_0^\infty(\R^+)$ supported in $[\frac12,2]$ such that
$\sum\limits_{N\in2^\Z}\beta(\frac{\rho}N)=1$ for all $\rho\in R^+$.
Let $\beta_N(\rho)=\beta(\frac{\rho}N)$ and $\tilde{\beta}_N$ be
similar to $\beta_N$. For simplicity, we assume $u_1=0$. Then we can
write
\begin{equation}\label{4.13}
\begin{split} u(x,t)=&\sum_{M\in2^{\Z}}\Big\{Y_{0,1}(\theta)\mathcal{H}_{\nu(0)}\big[\cos(
t\rho)b^0_{0,1}(\rho)\beta_M(\rho)\big](r)\\&\qquad\quad+\sum_{N\in2^{\N}}\sum_{k}\tilde{\beta}_N(k)\sum_{\ell=1}^{d(k)}Y_{k,\ell}(\theta)\mathcal{H}_{\nu(k)}\big[\cos(
t\rho)b^0_{k,\ell}(\rho)\beta_M(\rho)\big](r)\Big\}\\:=&u_{<}(x,t)+u_{\geq}(x,t).
\end{split}
\end{equation}

Without loss of the generality, it suffices to estiamte
$u_{\geq}(x,t)$. By Lemma \ref{square-expression}, Lemma
\ref{square-expression2} and the scaling argument, we show that for
$2\leq q, r$ and $r<\infty$
\begin{equation}\label{4.14}
\begin{split}
&\|u_{\geq}(t,x)\|^2_{L^q_t(\R;L^r_x(\R^n))}\\&\lesssim
\sum_{M\in2^{\Z}}\sum_{N\in2^{\N}}\Big\|\sum_{k}\tilde{\beta}_N(k)\sum_{\ell=1}^{d(k)}Y_{k,\ell}(\theta)\mathcal{H}_{\nu(k)}\big[\cos(
t\rho)b^0_{k,\ell}(\rho)\beta_M(\rho)\big](r)\Big\|^2_{L^q_t(\R;L^r_x(\R^n))}\\&\lesssim
\sum_{M\in2^{\Z}}M^{2(n-\frac 1q-\frac
nr)}\sum_{N\in2^{\N}}\Big\|\sum_{k}\tilde{\beta}_N(k)\sum_{\ell=1}^{d(k)}Y_{k,\ell}(\theta)\mathcal{H}_{\nu(k)}\big[\cos(
t\rho)b^0_{k,\ell}(M{\rho})\beta(\rho)\big](r)\Big\|^2_{L^q_t(\R;L^r_x(\R^n))}.
\end{split}
\end{equation}
$\bullet$ {\bf Case 1:} $n\geq4$. we have by interpolation
\begin{equation}\label{4.15}
\begin{split}
&\Big\|\sum_{k}\tilde{\beta}_N(k)\sum_{\ell=1}^{d(k)}Y_{k,\ell}(\theta)\mathcal{H}_{\nu(k)}\big[\cos(
t\rho)b^0_{k,\ell}(M{\rho})\beta(\rho)\big](r)\Big\|_{L^q_t(\R;L^r_x(\R^n))}\\
&\lesssim\Big\|\sum_{k}\tilde{\beta}_N(k)\sum_{\ell=1}^{d(k)}Y_{k,\ell}(\theta)\mathcal{H}_{\nu(k)}\big[\cos(
t\rho)b^0_{k,\ell}(M{\rho})\beta(\rho)\big](r)\Big\|^\lambda_{L^2_t(\R;L^{\gamma_0}_x(\R^n))}\\&\qquad\times
\Big\|\sum_{k}\tilde{\beta}_N(k)\sum_{\ell=1}^{d(k)}Y_{k,\ell}(\theta)\mathcal{H}_{\nu(k)}\big[\cos(
t\rho)b^0_{k,\ell}(M{\rho})\beta(\rho)\big](r)\Big\|^{1-\lambda}_{L^{\infty}_t(\R;L^2_x(\R^n))},
\end{split}
\end{equation}
where
\begin{equation}\label{4.16}
\begin{split}
\frac 1q=\frac\lambda 2+\frac{1-\lambda}\infty,\qquad
\frac1r=\frac\lambda {\gamma_0}+\frac{1-\lambda}2,\qquad
\frac1{\gamma_0}=\frac q2(\frac1r+\frac1q-\frac12).
\end{split}
\end{equation}
Since  $(q,r)\in\Lambda$, one has
$\frac{2(n-1)}{n-2}<\gamma_0\leq\frac{2(n-1)}{n-3}$. By
\eqref{Bernstein} and the argument in proving Proposition \ref{pro},
one has
\begin{equation}\label{4.17}
\begin{split}
&\Big\|\sum_{k}\tilde{\beta}_N(k)\sum_{\ell=1}^{d(k)}Y_{k,\ell}(\theta)\mathcal{H}_{\nu(k)}\big[\cos(
t\rho)b^0_{k,\ell}(M{\rho})\beta(\rho)\big](r)\Big\|_{L^2_t(\R;L^{\frac{2(n-1)}{n-2}+}_{x}(\R^n))}\\&\lesssim
N^{\frac12+}\Big\|\Big(\sum_{k}\sum_{\ell=1}^{d(k)}\big|\tilde{\beta}_N(k)\mathcal{H}_{\nu(k)}\big[\cos(
t\rho)b^0_{k,\ell}(M{\rho})\beta(\rho)\big](r)\big|^2\Big)^{\frac12}\Big\|_{L^2_t(\R;L^{\frac{2(n-1)}{n-2}+}_{r^{n-1}\mathrm{d}r}(\R^+))}\\&\lesssim
N^{\frac12+}\Big\|\Big(\sum_{k}\sum_{\ell=1}^{d(k)}\big|\tilde{\beta}_N(k)b^0_{k,\ell}(M{\rho})\beta(\rho)\big|^2\Big)^{\frac12}\Big\|_{L^2_\rho}.
\end{split}
\end{equation}
On the other hand, using the endpoint Strichartz estimate in Lemma
\ref{stri}, we have
\begin{equation}\label{4.18}
\begin{split}
&\Big\|\sum_{k}\tilde{\beta}_N(k)\sum_{\ell=1}^{d(k)}Y_{k,\ell}(\theta)\mathcal{H}_{\nu(k)}\big[\cos(
t\rho)b^0_{k,\ell}(M{\rho})\beta(\rho)\big](r)\Big\|_{L^2_t(\R;L^{\frac{2(n-1)}{n-3}}_{x}(\R^n))}\\&\lesssim
\Big\|\Big(\sum_{k}\sum_{\ell=1}^{d(k)}\big|\tilde{\beta}_N(k)b^0_{k,\ell}(M{\rho})\beta(\rho)\big|^2\Big)^{\frac12}\Big\|_{L^2_\rho}.
\end{split}
\end{equation}
Therefore, we obtain by interpolation
\begin{equation}\label{4.19}
\begin{split}
&\Big\|\sum_{k}\tilde{\beta}_N(k)\sum_{\ell=1}^{d(k)}Y_{k,\ell}(\theta)\mathcal{H}_{\nu(k)}\big[\cos(
t\rho)b^0_{k,\ell}(M{\rho})\beta(\rho)\big](r)\Big\|_{L^2_t(\R;L^{\gamma_0}_x(\R^n))}\\&\lesssim
N^{(n-1)(\frac1{\gamma_0}-\frac{n-3}{2(n-1)})+}\Big\|\Big(\sum_{k}\sum_{\ell=1}^{d(k)}\big|\tilde{\beta}_N(k)b^0_{k,\ell}(M{\rho})\beta(\rho)\big|^2\Big)^{\frac12}\Big\|_{L^2_\rho}.
\end{split}
\end{equation}
By Lemma \ref{Hankel3}, we have
\begin{equation*}
\begin{split}
\big\|[\mathcal{H}_{\nu(k)}a_{k,\ell}](r)\big\|_{L^2_{r^{n-1}\mathrm{d}r}}=\|a_{k,\ell}(\rho)\|_{L^2_{\rho^{n-1}\mathrm{d}\rho}}.
\end{split}
\end{equation*}
We are in sprit of energy estimate to obtain
\begin{equation}\label{4.20}
\begin{split}
&\Big\|\sum_{k}\tilde{\beta}_N(k)\sum_{\ell=1}^{d(k)}Y_{k,\ell}(\theta)\mathcal{H}_{\nu(k)}\big[\cos(
t\rho)b^0_{k,\ell}(M{\rho})\beta(\rho)\big](r)\Big\|_{L^\infty_t(\R;L^2_{r^{n-1}\mathrm{d}r}(\R^+;L^2(\mathbb{S}^{n-1})))}\\&\lesssim
\Big\|\Big(\sum_{k}\sum_{\ell=1}^{d(k)}\big|\tilde{\beta}_N(k)\mathcal{H}_{\nu(k)}\big[\cos(
t\rho)b^0_{k,\ell}(M{\rho})\beta(\rho)\big](r)\big|^2\Big)^{\frac12}\Big\|_{L^\infty_t(\R;L^{2}_{r^{n-1}\mathrm{d}r}(\R^+))}\\&\lesssim
\Big(\sum_{k}\sum_{\ell=1}^{d(k)}\tilde{\beta}_N(k)\big\|b^0_{k,\ell}(M{\rho})\beta(\rho)\big\|_{L^2_{\rho}}^2\Big)^{\frac12}.
\end{split}
\end{equation}
Combining \eqref{4.14},\eqref{4.15}, \eqref{4.19}, and \eqref{4.20},
we have
\begin{equation}\label{4.21}
\begin{split}
&\|u_{\geq}(t,x)\|^2_{L^q_t(\R;L^r_x(\R^n))}\\&\lesssim
\sum_{M\in2^{\Z}}M^{2(n-\frac 1q-\frac
nr)}\sum_{N\in2^{\N}}N^{2\lambda(n-1)(\frac1{\gamma_0}-\frac{n-3}{2(n-1)})+}\sum_{k}\sum_{\ell=1}^{d(k)}\tilde{\beta}_N(k)\big\|b^0_{k,\ell}(M{\rho})\beta(\rho)\big\|_{L^2_{\rho}}^2
\\&\lesssim
\sum_{M\in2^{\Z}}M^{2(n-\frac 1q-\frac
nr)}\sum_{N\in2^{\N}}N^{2[\frac2q+(n-1)(\frac1r-\frac12)]+}\sum_{k}\sum_{\ell=1}^{d(k)}\tilde{\beta}_N(k)\big\|b^0_{k,\ell}(M{\rho})\beta(\rho)\big\|_{L^2_{\rho}}^2.
\end{split}
\end{equation}
By making use of Lemma \ref{orthogonality},
$s=n(\frac12-\frac1r)-\frac1q$, and
$\bar{s}=(1+\epsilon)(\frac2q-(n-1)(\frac12-\frac1r))$, we get
\begin{equation}\label{4.22}
\begin{split}
\|u_{\geq}(t,x)\|_{L^q_t(\R;L^r_x(\R^n))}\lesssim
\|\langle\Omega\rangle^{\bar{s}}u_0\|_{\dot H^s}.
\end{split}
\end{equation}
$\bullet$ {\bf Case 2:} $n=3$. Since the endpoint Strichartz
estimate fails, the above argument breaks down. By the
interpolation, we have
\begin{equation}\label{4.23}
\begin{split}
&\Big\|\sum_{k}\tilde{\beta}_N(k)\sum_{\ell=1}^{d(k)}Y_{k,\ell}(\theta)\mathcal{H}_{\nu(k)}\big[\cos(
t\rho)b^0_{k,\ell}(M{\rho})\beta(\rho)\big](r)\Big\|_{L^q_t(\R;L^r_x(\R^3))}\\
&\lesssim\Big\|\sum_{k}\tilde{\beta}_N(k)\sum_{\ell=1}^{d(k)}Y_{k,\ell}(\theta)\mathcal{H}_{\nu(k)}\big[\cos(
t\rho)b^0_{k,\ell}(M{\rho})\beta(\rho)\big](r)\Big\|^\lambda_{L^2_t(\R;L^{4+}_x(\R^3))}\\&\qquad
\times\Big\|\sum_{k}\tilde{\beta}_N(k)\sum_{\ell=1}^{d(k)}Y_{k,\ell}(\theta)\mathcal{H}_{\nu(k)}\big[\cos(
t\rho)b^0_{k,\ell}(M{\rho})\beta(\rho)\big](r)\Big\|^{1-\lambda}_{L^{q_0}_t(\R;L^{r_0}_x(\R^3))},
\end{split}
\end{equation}
where
\begin{equation}\label{4.24}
\begin{split}
\frac 1q=\frac\lambda 2+\frac{1-\lambda}{q_0},\qquad
\frac1r=\frac\lambda {4+}+\frac{1-\lambda}{r_0},\qquad
\frac{1}2=\frac 1{q_0}+\frac{1}{r_0},\quad r_0\neq \infty.
\end{split}
\end{equation}

By \eqref{Bernstein} and the argument in proving Proposition
\ref{pro}, one has
\begin{equation}\label{4.25}
\begin{split}
&\Big\|\sum_{k}\tilde{\beta}_N(k)\sum_{\ell=1}^{d(k)}Y_{k,\ell}(\theta)\mathcal{H}_{\nu(k)}\big[\cos(
t\rho)b^0_{k,\ell}(M{\rho})\beta(\rho)\big](r)\Big\|_{L^2_t(\R;L^{4+}_{x}(\R^3))}\\&\lesssim
N^{\frac12+}\Big\|\Big(\sum_{k}\sum_{\ell=1}^{d(k)}\big|\tilde{\beta}_N(k)\mathcal{H}_{\nu(k)}\big[\cos(
t\rho)b^0_{k,\ell}(M{\rho})\beta(\rho)\big](r)\big|^2\Big)^{\frac12}\Big\|_{L^2_t(\R;L^{4+}_{r^{n-1}\mathrm{d}r}(\R^+))}\\&\lesssim
N^{\frac12+}\Big\|\Big(\sum_{k}\sum_{\ell=1}^{d(k)}\big|\tilde{\beta}_N(k)b^0_{k,\ell}(M{\rho})\beta(\rho)\big|^2\Big)^{\frac12}\Big\|_{L^2_\rho}.
\end{split}
\end{equation}
On the other hand, by the  Strichartz estimate with $(q_0, r_0)$ in
Lemma \ref{stri}, we have
\begin{equation}\label{4.26}
\begin{split}
&\Big\|\sum_{k}\tilde{\beta}_N(k)\sum_{\ell=1}^{d(k)}Y_{k,\ell}(\theta)\mathcal{H}_{\nu(k)}\big[\cos(
t\rho)b^0_{k,\ell}(M{\rho})\beta(\rho)\big](r)\Big\|_{L^{q_0}_t(\R;L^{r_0}_{x}(\R^3))}\\&\lesssim
\Big\|\Big(\sum_{k}\sum_{\ell=1}^{d(k)}\big|\tilde{\beta}_N(k)b^0_{k,\ell}(M{\rho})\beta(\rho)\big|^2\Big)^{\frac12}\Big\|_{L^2_\rho}.
\end{split}
\end{equation}
This together with \eqref{4.14}, \eqref{4.23} and \eqref{4.25}
yields that
\begin{equation}\label{4.27}
\begin{split}
&\|u_{\geq}(t,x)\|^2_{L^q_t(\R;L^r_x(\R^n))}\\&\lesssim
\sum_{M\in2^{\Z}}M^{2(n-\frac 1q-\frac
nr)}\sum_{N\in2^{\N}}N^{\lambda+}\sum_{k}\sum_{\ell=1}^{d(k)}\tilde{\beta}_N(k)\big\|b^0_{k,\ell}(M{\rho})\beta(\rho)\big\|_{L^2_{\rho}}^2
\\&\lesssim
\sum_{M\in2^{\Z}}M^{2(n-\frac 1q-\frac
nr)}\sum_{N\in2^{\N}}N^{(4+\epsilon)[\frac1q+\frac1r-\frac12]}\sum_{k}\sum_{\ell=1}^{d(k)}\tilde{\beta}_N(k)\big\|b^0_{k,\ell}(M{\rho})\beta(\rho)\big\|_{L^2_{\rho}}^2.
\end{split}
\end{equation}
 Since $s=n(\frac12-\frac1r)-\frac1q$, by the scaling argument, Lemma
\ref{orthogonality}, we  get
\begin{equation}\label{4.28}
\begin{split}
\|u_{\geq}(t,x)\|_{L^q_t(\R;L^r_x(\R^n))}\lesssim
\|\langle\Omega\rangle^{\bar{s}}u_0\|_{\dot H^s}.
\end{split}
\end{equation}
Moreover, for $q=2, 4<r<\infty$, \eqref{4.14} and the ``Bernstein"
inequality \eqref{Bernstein} imply that
\begin{equation*}
\begin{split}
&\|u_{\geq}(t,x)\|^2_{L^2_t(\R;L^r_x(\R^n))}
\\ \lesssim&
\sum_{M\in2^{\Z}}M^{2(n-\frac 12-\frac
nr)}\sum_{N\in2^{\N}}\Big\|\sum_{k}\tilde{\beta}_N(k)\sum_{\ell=1}^{d(k)}Y_{k,\ell}(\theta)\mathcal{H}_{\nu(k)}\big[\cos(
t\rho)b^0_{k,\ell}(M{\rho})\beta(\rho)\big](r)\Big\|^2_{L^2_t(\R;L^r_x(\R^n))}\\
\lesssim&\sum_{M\in2^{\Z}}M^{2(n-\frac 12-\frac
nr)}\sum_{N\in2^{\N}}
N^{2\bar{s}(r)}\Big\|\Big(\sum_{k}\sum_{\ell=1}^{d(k)}\big|\tilde{\beta}_N(k)\mathcal{H}_{\nu(k)}\big[\cos(
t\rho)b^0_{k,\ell}(M{\rho})\beta(\rho)\big](r)\big|^2\Big)^{\frac12}\Big\|^2_{L^2_t(\R;L^{r}_{r^{n-1}\mathrm{d}r}(\R^+))}\\\lesssim
& \|\langle\Omega\rangle^{\bar{s}(r)}u_0\|^2_{\dot H^s}.
\end{split}
\end{equation*}
Combining this with \eqref{4.22} and \eqref{4.28}, we complete the
proof of Theorem \ref{thm}.

\section{Proof of Theorem \ref{thm1}}
To prove Theorem \ref{thm1}, we first use the inhomogeneous
Strichartz estimates for the wave equation without potential in
\cite{H,O} and the arguments in \cite{PSS} to prove an inhomogeneous
Strichartz estimates for the wave equation with inverse-square
potential.

\begin{proposition}[Inhomogeneous Strichartz estimates]\label{inh}
Let $\widetilde{\Box}=\partial_t^2+A_{\nu}$ and let $v$ solve the
inhomogeneous wave equation $\widetilde{\Box} v=h$ in $\R\times\R^n$
with zero initial data. If $\nu>\max\{\frac{n-2}2-\frac n{q_0},
\frac n{r_0}-\frac{n-2}2-2\}$, then
\begin{equation}\label{5.1}
\begin{split}
\|v\|_{L^{q_0}_{t,x}(\R\times\R^n)}\lesssim
\|h\|_{L^{r_0}_{t,x}(\R\times\R^n)},
\end{split}
\end{equation}
where $q_0=(p-1)(n+1)/2$ and $r_0=(n+1)(p-1)/(2p)$ with
$p_{\text{h}}<p<p_{\text{conf}}$.
\end{proposition}
\begin{proof}
By the continuity property of $\mathcal{K}^0_{\nu,\lambda}$ in Lemma
\ref{continuous}, it follows that
\begin{equation}\label{5.2}
\begin{split}
&\|v\|_{L^{q_0}_{t,x}(\R\times\R^n)}\leq
\|\mathcal{K}^0_{\nu,\lambda}\|_{q_0\rightarrow q_0}
\|\mathcal{K}^0_{\lambda,\nu}v\|_{L^{q_0}_{t,x}(\R\times\R^n)}.
\end{split}
\end{equation}
Noting that \eqref{2.21} with $k=0$, one has
$\mathcal{K}^0_{\lambda,\nu}h=\mathcal{K}^0_{\lambda,\nu}\widetilde{\Box}v=\Box\mathcal{K}^0_{\lambda,\nu}v$.
We recall the inhomogeneous Strichartz estimates, for
$\frac1{r}-\frac1{q}=\frac2{n+1}$ and
$\frac{2n}{n-1}<q<\frac{2(n+1)}{n-1}$,
\begin{equation*}
\begin{split}
\|u\|_{L^{q}_{t,x}(\R\times\R^n)}\leq C\|
(\partial_{tt}-\Delta)u\|_{L^{r}_{t,x}(\R\times\R^n)},
\end{split}
\end{equation*}
which was shown by Harmse\cite{H} and Oberlin \cite{O}. Thus we
obtain
\begin{equation}\label{5.3}
\begin{split}
&\|v\|_{L^{q_0}_{t,x}(\R\times\R^n)}\leq C
\|\mathcal{K}^0_{\nu,\lambda}\|_{q_0\rightarrow q_0}
\|\mathcal{K}^0_{\lambda,\nu}
h\|_{L^{r_0}_{t,x}(\R\times\R^n)}\\&\lesssim
\|\mathcal{K}^0_{\nu,\lambda}\|_{q_0\rightarrow
q_0}\|\mathcal{K}^0_{\lambda,\nu}\|_{r_0\rightarrow r_0}\|
h\|_{L^{r_0}_{t,x}(\R\times\R^n)},
\end{split}
\end{equation}
where we use the facts that $\frac1{r_0}-\frac1{q_0}=\frac2{n+1}$
and $\frac{2n}{n-1}<q_0<\frac{2(n+1)}{n-1}$.
\end{proof}
Now we are in position to prove Theorem \ref{thm1}. Define the
solution map
\begin{equation*}
\begin{split}
\Phi(u)&=\cos\big(t\sqrt{P_a}\big)u_0(x)+\frac{\sin\big(t\sqrt{P_a}\big
)}{\sqrt{P_a}}u_1(x)+\int_0^t\frac{\sin\big((t-s)\sqrt{P_a}\big)}{\sqrt{P_a}}F(u(s,x))\mathrm{d}s
\\&:=u_{\text{hom}}+u_{\text{inh}},
\end{split}
\end{equation*}
on the complete metric space $X$
$$X=\big\{u: u\in C_t(\dot H^{s_c})\cap L^{q_0}_{t,x},
\|u\|_{L^{q_0}_{t,x}}\leq 2C\epsilon\big\}$$ with the metric
$d(u_1,u_2)=\|u_1-u_2\|_{L^{q_0}_{t,x}}$,
 where $P_a=-\Delta+\frac{a}{|x|^2}$
with $a$ satisfying \eqref{1.7}, and $\epsilon\leq\epsilon(p)$ is as
in \eqref{1.8}.

From Lemma \ref{stri1} and \eqref{1.5}, we obtain
\begin{equation}\label{5.4}
\begin{split}
\|u_{\text{hom}}\|_{C_t(\dot H^{s_c})\cap L^{q_0}_{t,x}}\leq
C\big(\|u_0\|_{\dot H^{s_c}}+\|u_1\|_{\dot H^{s_c-1}}\big)\leq
C\epsilon.
\end{split}
\end{equation}
By Proposition \ref{inh} and the inhomogeneous version of the
 Strichartz estimates \eqref{1.2}, one has
\begin{equation}\label{5.5}
\begin{split}
\|u_{\text{inh}}\|_{C_t(\dot H^{s_c})\cap L^{q_0}_{t,x}}\leq
C\|F(u)\|_{L^{r_0}_{t,x}}\leq C\|u\|^{p}_{L^{q_0}_{t,x}}\leq
C^2(C\epsilon)^{p-1}\epsilon\leq C\epsilon.
\end{split}
\end{equation}
A similar argument as above leads to
\begin{equation}\label{5.6}
\begin{split}
\|\Phi(u_1)-\Phi(u_2)\|_{L^{q_0}_{t,x}}\leq&
C\|F(u_1)-F(u_2)\|_{L^{r_0}_{t,x}}\\\leq&
C^2(C\epsilon)^{p-1}\|u_1-u_2\|_{L^{q_0}_{t,x}}\leq
\frac12\|u_1-u_2\|_{L^{q_0}_{t,x}}.
\end{split}
\end{equation}
Therefore, the solution map $\Phi$ is a contraction map on $X$ under
the matric $d(u_1,u_2)=\|u_1-u_2\|_{L^{q_0}_{t,x}}$. The standard
contraction argument gives the proof.

\section{Appendix: The Proof of Lemma \ref{square-expression}}
We will apply the H\"ormander's technique to showing a weak-type
$(1,1)$ estimate for the multiplier operators with respect to the
Hankel transform.

The multiplier operators associated with the Hankel transform are
defined by
\begin{equation}\label{6.1}
[L_jf](r)=\int_0^\infty(r\rho)^{-\frac{n-2}2}J_{\nu}(r\rho)[\mathcal{H}_{\nu}f](\rho)\beta_j(\rho)\mathrm{d}\omega(\rho),\quad
j\in\Z
\end{equation}
where
\begin{equation}\label{6.2}
(\mathcal{H}_{\nu}f)(\rho)=\int_0^\infty(r\rho)^{-\frac{n-2}2}J_{\nu}(r\rho)f(r)\mathrm{d}\omega(r),\quad
\text{and}\quad \mathrm{d}\omega(\rho)=\rho^{n-1}\mathrm{d}\rho.
\end{equation}
Since $\mathcal{H}_{\nu}=\mathcal{H}^{-1}_{\nu}$, we have
$\mathcal{H}_{\nu}[L_jf]=\beta_j(\rho)[\mathcal{H}_{\nu}f]$. We
first claim that
\begin{equation}\label{6.3}
\big\|\big(\sum_{j\in\Z}|L_jf|^2\big)^{\frac12}\big\|_{L^p(\omega)}\sim\big\|\sum_{j\in\Z}L_jf\big\|_{L^p(\omega)}\sim\|f\|_{L^p(\omega)}.
\end{equation}
This implies Lemma \ref{square-expression}, by choosing
$f=\mathcal{H}_{\nu}[\cos(t\rho)b_{k,\ell}^0(\rho)]$.  To show
\eqref{6.3}, we need the following
\begin{lemma}\label{square}
Let $f\in L^p(\omega)$, $1<p<\infty$. Then there exists a constant
$C_p$ such that
\begin{equation}\label{6.4}
\big\|\big(\sum_{j\in\Z}|{L}_jf|^2\big)^{\frac12}\big\|_{L^p(\omega)}\leq
C_p\|f\|_{L^p(\omega)}.
\end{equation}
\end{lemma}
We postpone the proof for a moment. By duality, one has
\begin{equation*}
\|f\|_{L^p(\omega)}=\sup_{\|g\|_{L^{p'}(\omega)}\leq1}\int_0^\infty
f(r)\overline{g}(r)\mathrm{d}\omega(r).
\end{equation*}
By Lemma \ref{Hankel3}, we observe that
\begin{equation*}
\begin{split}
\int_0^\infty
f(r)\overline{g}(r)\mathrm{d}\omega(r)&=\sum_{j,j'\in\Z}\int_0^\infty
\mathcal{H}_{\nu}[L_jf](\rho){\mathcal{H}_{\nu}[L_{j'}\overline{g}]}(\rho)\mathrm{d}\omega(\rho)
\\&=\sum_{j,j'\in\Z}\int_0^\infty \beta_j(\rho)\beta_{j'}(\rho)[\mathcal{H}_{\nu}f](\rho)[\mathcal{H}_{\nu}\overline{g}](\rho)\mathrm{d}\omega(\rho)
\\&\leq
C\sum_{j\in\Z}\int_0^\infty\beta_j(\rho)\beta_{j}(\rho)[\mathcal{H}_{\nu}f](\rho)[\mathcal{H}_{\nu}\overline{g}](\rho)\mathrm{d}\omega(\rho).
\end{split}
\end{equation*}
This implies that
\begin{equation}\label{6.5}
\begin{split}
\int_0^\infty f(r)\overline{g}(r)\mathrm{d}\omega(r)\leq
C\sum_{j\in\Z}\int_0^\infty [L_jf](r)[L_jg](r)\mathrm{d}\omega(r).
\end{split}
\end{equation}
Hence by Lemma \ref{square}, we obtain
\begin{equation}
\begin{split} \|f\|_{L^p(\omega)}&\leq
C\sup_{\|g\|_{L^{p'}(\omega)}\leq1}\big\|\big(\sum_{j\in\Z}|{L}_jf|^2\big)^{\frac12}\big\|_{L^p(\omega)}\big\|\big(\sum_{j\in\Z}|{L}_jg|^2\big)^{\frac12}\big\|_{L^{p'}(\omega)}
\\&\leq C\big\|\big(\sum_{j\in\Z}|{L}_jf|^2\big)^{\frac12}\big\|_{L^p(\omega)}.
\end{split}
\end{equation}
This together with \eqref{6.4} gives \eqref{6.3}. When $p=2$, we
have by Lemma \ref{Hankel3}
\begin{equation*}
\begin{split}
&\Big\|\big(\sum_{j\in\Z}|{L}_jf|^2\big)^{\frac12}\Big\|^2_{L^2(\omega)}=\sum_{j\in\Z}\big\|{L}_jf\big\|^2_{L^2(\omega)}
=\int_0^\infty\sum_{j\in\Z}|\beta_j(\rho)|^2|\mathcal{H}_{\nu}f|^2\mathrm{d}\omega(\rho)\leq
C\|f\|_{L^2(\omega)}.
\end{split}
\end{equation*}
Define the operator $S(f)$ by $f~\mapsto~\{L_j f\}_{j\in\Z}$, then
$\|S(f)\|_{L^2(\omega;\ell^2(\Z))}\leq C\|f\|_{L^2(\omega)}$.

To show \eqref{6.4}, it suffices to prove
\begin{equation}\label{6.7}
\begin{split}
\|S(f)\|_{L^{1,\infty}(\omega;\ell^2(\Z))}\leq C\|f\|_{L^1(\omega)},
\end{split}
\end{equation}
where $L^{1,\infty}(\omega)$ denotes the weak-$L^1(\omega)$. Define
the generalized convolution $f\# g$ as
\begin{equation}\label{6.8}
\begin{split}
f\# g(x)=\int_0^\infty(\tau_xf)(y)g(y)\mathrm{d}\omega(y), \quad
x\in \R^+,
\end{split}
\end{equation}
where $f,g\in L^1(\omega)$, the Hankel translation $\tau_xf$ is
denoted to be
\begin{equation}\label{6.9}
\begin{split}
(\tau_xf)(y)=\int_0^\infty K_{\nu}(x,y,z)f(z)\mathrm{d}\omega(z),
\quad x,y\in \R^+,
\end{split}
\end{equation}
and
\begin{equation}\label{6.10}
\begin{split}
K_{\nu}(x,y,z)=\int_0^\infty(xt)^{-\frac{n-2}2}J_{\nu}(xt)(yt)^{-\frac{n-2}2}J_{\nu}(yt)(zt)^{-\frac{n-2}2}J_{\nu}(zt)\mathrm{d}\omega(t),
\quad x,y,z\in \R^+.
\end{split}
\end{equation}
Then $\mathcal{H}_{\nu}[f\#
g]=\mathcal{H}_{\nu}(f)\mathcal{H}_{\nu}(g)$. Moreover, we have
$L_jf=k_j\# f$ with $k_j=\mathcal{H}_{\nu}(\beta_j)$. Taking into
account the fact that $(\tau_x f)(y)=(\tau_y f)(x)$ and  Theorem 2.4
in \cite{CW}, it suffices to prove the Hankel version of the
well-known H\"ormander condition
\begin{equation*}
\begin{split}
\int_{|x-y_0|>2|y-y_0|}\big(\sum_{j\in\Z}\big|\tau_{y}k_j(x)-\tau_{y_0}k_j(x)\big|^2\big)^{\frac12}\mathrm{d}\omega(x)\leq
C,
\end{split}
\end{equation*}
where $C$ is independent of $y, y_0$. This is implied by
\begin{equation*}
\begin{split}
\sum_{j\in\Z}\int_{|x-y_0|>2|y-y_0|}\big|\tau_{y}k_j(x)-\tau_{y_0}k_j(x)\big|\mathrm{d}\omega(x)\leq
C,
\end{split}
\end{equation*}
which can be proved by the arguments in \cite{BM,GS}. \vskip0.5cm

{\bf Acknowledgments:}\quad  The authors would like to express their
gratitude to Professor S. Shao for his helpful discussions and
 the anonymous referee
 for their invaluable comments and
suggestions. The authors were partly supported by the NSF of China
(No.11171033, No.11231006) and by Beijing Center of Mathematics and
Information Science.


\end{document}